\documentclass[11pt]{amsart}
\usepackage{amsmath,amsthm,amssymb,latexsym}
\newtheorem{teo}{Theorem}
\newtheorem*{teo*}{Theorem}
\newtheorem{co}[teo]{Corollary}
\newtheorem{lemma}[teo]{Lemma}
\newtheorem{prop}[teo]{Proposition}
\newtheorem*{sublema*}{Sublemma}

\theoremstyle{definition}

\theoremstyle{remark}

\newtheorem*{remark}{Remark}

\newcommand{\Hu}{{\mathcal H}^1}
\newcommand{\Rn}{{\mathbb R}^d}
\newcommand{\Rd}{{\mathbb R}^2}
\newcommand{\ep}{\varepsilon}
\newcommand{\N}{\mathbb{N}}

\newcommand{\op}{\operatorname{op}}
\newcommand{\id}{\operatorname{id}}
\newcommand{\pv}{\operatorname{p.v.}}
\newcommand{\cc}{{\mathcal C}}

\newcommand{\CC}{{\mathcal C}}

\newcommand{\diam}{{\rm diam}}

\newcommand{\ds}{\displaystyle }

\newcommand{\rf}[1]{{(\ref{#1})}}

\newcommand{\supp}{{\rm supp}}

\newcommand{\vphi}{{\varphi}}
\newcommand{\ve}{{\varepsilon}}

\newcommand{\wt}[1]{{\widetilde{#1}}}

\newcommand{\rest}{{\lfloor}}

\begin{document}
\title{Capacities associated with Calder\'on-Zygmund kernels}

\author{V. Chousionis, J. Mateu, L. Prat and X. Tolsa}
\thanks{Most of this work had been carried out in the first semester of 2011 while V.C was visiting the Centre de Recerca Matem\`atica in Barcelona and he feels grateful for the hospitality. V.C was supported by the Academy of Finland. J.M and L.P are supported by grants 2009SGR-000420 (Generalitat de Catalunya) and
MTM2010-15657 (Spain). X.T is supported by grants 2009SGR-000420 (Generalitat de Catalunya) and MTM2010-16232
(Spain).}

\address{Vasilis Chousionis.  Departament de
Ma\-te\-m\`a\-ti\-ques, Universitat Aut\`onoma de Bar\-ce\-lo\-na, Catalonia}

\email{vasileios.chousionis@helsinki.fi}

\address{Joan Mateu.  Departament de
Ma\-te\-m\`a\-ti\-ques, Universitat Aut\`onoma de Bar\-ce\-lo\-na, Catalonia}

\email{mateu@mat.uab.cat}

\address{Laura Prat.  Departament de
Ma\-te\-m\`a\-ti\-ques, Universitat Aut\`onoma de Bar\-ce\-lo\-na, Catalonia}

\email{laurapb@mat.uab.cat}
\address{Xavier Tolsa. Instituci\'{o} Catalana de Recerca
i Estudis Avan\c{c}ats (ICREA) and Departament de
Ma\-te\-m\`a\-ti\-ques, Universitat Aut\`onoma de Bar\-ce\-lo\-na.
08193 Barcelona, Catalonia} \email{xtolsa@mat.uab.cat}

\maketitle
\begin{abstract}
Analytic capacity is associated with the Cauchy kernel $1/z$ and the $L^\infty$-norm. 
For $n\in\mathbb{N}$, one has likewise capacities related to the kernels 
$K_i(x)=x_i^{2n-1}/|x|^{2n}$, $1\le i\le 2$, $x=(x_1,x_2)\in\Rd$. The main result of this paper states that the capacities associated with the vectorial kernel $(K_1, K_2)$ are 
comparable to analytic capacity.
\end{abstract}
\vspace{1cm}

The analytic capacity of a compact subset $E$ of the plane is
defined by
$$ \gamma(E)=\sup|f'(\infty)|$$
where the supremum is taken over those analytic functions in
$\mathbb{C}\setminus E$ such that $|f(z)|\leq 1$ for all  $z\in
\mathbb{C}\setminus E$.  Sets of zero analytic capacity are exactly
the removable sets for bounded analytic functions, as shown by Ahlfors, and thus $\gamma(E)$ quantifies the non-removability of $E$.
Early work on analytic capacity used basically one complex variable
methods (see, e.g., \cite{Ahlfors}, \cite{garnett} and \cite{Vi}).
Analytic capacity may be written as
\begin{equation}\label{acap2}
\gamma(E)=\sup |\langle T,1\rangle|,
\end{equation}
where the supremum is taken over all complex distributions $T$
supported on $E$ whose Cauchy potential $f=1/z * T$ is in the closed
unit ball of  $L^\infty(\mathbb{C})$. 
Expression \eqref{acap2} shows that analytic capacity is formally an
analogue of classical logarithmic capacity, in which the logarithmic
kernel has been replaced by the complex kernel $1/z$. This suggests
that real variables techniques could help in studying analytic
capacity, in spite of the fact that the Cauchy kernel is complex. In
fact, significant progress in the understanding of analytic capacity
was achieved when real variables methods were
systematically used (\cite{C}, \cite{davidvitushkin},  \cite{mmv},
\cite{mtv}, \cite{semiad} and~\cite{bilipschitz}), in particular the
Calder\'{o}n-Zygmund theory of the Cauchy singular integral. 

Recall that for a Borel set $E$ with finite length, $0<\mathcal{H}^1(E)<\infty$, David and L\'eger (see \cite{Leger}) proved that the $L^2(\mathcal{ H}^1|E)-$boundedness of the 
singular integral associated with the Cauchy kernel (or even with one of its coordinate parts $x_1/|x|^2$, $x_2/|x|^2$, $x=(x_1,x_2)\in\Rd$) implies that $E$ is rectifiable. 
We recall that a set in $\Rd$ is rectifiable if it is contained, up to an $\mathcal{H}^1$-negligible set, in a countable union of $1$-dimensional Lipschitz graphs. 
In \cite{cmpt} we extended this result to any kernel of the form $x_i^{2n-1}/|x|^{2n}$, $i=1,2$, $n\in\N,$ providing the first non-trivial examples of operators not 
directly related to the Cauchy transform whose $L^2-$boundedness implies rectifiability.

In this paper we introduce capacities associated with these kernels.
For $n\geq 1$, write $x=(x_1,x_2)\in\Rd$ and consider the kernels 
\begin{equation}\label{ourkernel}
K_1(x)=x_1^{2n-1}/|x|^{2n} \mbox{ and }K_2(x)=x_2^{2n-1}/|x|^{2n}.
\end{equation} 
For compact sets $E\subset\Rd$, we define 
\begin{equation*}
\gamma_n(E)=\sup|\langle T,1\rangle|,
\end{equation*}
the supremum taken over those real distributions $T$ supported on $E$ such
that for $i=1,2$, the potentials $K_i*T$
are in the unit ball of $L^\infty(\Rd)$.

We will show that the above defined capacity is comparable to analytic capacity, that is,

\begin{teo}\label{main}
There exists some positive constant $C$ such that for all compact sets $E\subset\Rd$,
\begin{equation*}
C^{-1}\gamma_n(E)\le\gamma(E)\le C\gamma_n(E).
\end{equation*}
\end{teo}

The main motivation to study these capacities is getting a better understanding of the relation between the operators whose $L^2-$boundedness implies rectifiability 
and the comparability of analytic capacity and the capacities related to the kernels of such operators. It is worth to mention that if one considers  
the kernel $k(x_1,x_2)=x_1x_2^2/|x|^4$, then the comparability result between analytic capacity and the capacity related to the kernel $k$ does not hold.
See Section \ref{sectionhuov} for more details.\newline

For our second main result, we turn to the higher dimensional setting. Motivated by the paper \cite{mpv}, we set $n=1$ and consider capacities in $\Rn$, 
associated with the kernels $x_i/|x|^2$, $1\le i\le d$. 

For a compact $E \subset \Rn$ set
$$
\Gamma(E)=\sup \left\{|\langle T ,1\rangle | \right\},
$$
where the supremum is taken over those real distributions $T$
supported on $E$ such that the vector field
$\displaystyle{\frac{x}{|x|^2}*T}$ is in the unit ball of
$L^\infty(\Rn, \Rn)$.  Notice that, for $d=2$, due to \cite{semiad}, $\Gamma(E)$ is comparable to the analytic capacity $\gamma(E)$ .
Finally, for $1\le k \le d$, set
\begin{equation}\label{Gammak}
\Gamma_{\hat k}(E)=\sup \left\{|\langle T ,1\rangle| :  \left\|\frac{x_i}{|x|^{2}}*T\right\|_\infty
\leq 1,\, 1 \leq i \leq d,\, i \neq k\right \}.
\end{equation}
Thus we require the boundedness of $d-1$ components of the vector
valued potential~$x/|x|^{2}*T$ with Riesz kernel of homogeneity
$-1$.
  
In the plane, an easy complex argument  (see \cite{mpv}) shows that
\begin{equation}\label{mark}
\gamma(E)\approx\Gamma_{\hat k}(E),\;\;k=1,2.
\end{equation}
However in higher dimensions, this is an open question and indeed very little is known about these capacities $\Gamma_{\hat k}$. The reason why  $\Gamma_{\hat k}$ is 
difficult to understand in higher dimensions is that boundedness of $d-1$ potentials does not provide any linear growth condition on the distribution $T$. 
Concretely, it is not true that boundedness of $x_i/|x|^{2}*T$, $1\le i\le d-1$, implies that for each cube $Q$ one has
\begin{equation}\label{basicgrowth}
|\langle T,\varphi_Q\rangle|\le Cl(Q),
\end{equation}
for each test function $\varphi_Q\in{\mathcal C}_0^\infty(Q)$ satisfying $\|\varphi_Q\|_\infty\le 1$ and $\|\nabla\varphi_Q\|_{\infty}\le l(Q)^{-1}$. 
See Section 5 of \cite{mpv} for some examples of such phenomenon. Here $l(Q)$ stands for the side length of $Q$.\newline
 
In \cite{mpv} it was shown that the capacities $\Gamma_{\hat k}(E)$ are finite. Moreover,  the following higher dimensional version of \eqref{mark} was also shown: for $d\ge 3$,
\begin{equation}\label{conjecture}
\Gamma(E)\approx\Gamma_{\hat k}(E),\;\;1\le k\le d,
\end{equation}
assuming an extra growth condition on the definition of the capacities $\Gamma_{\hat k}(E)$.  Naturally, the following open question appeared: is it true that \eqref{conjecture} holds without any growth condition on the definition of $\Gamma_{\hat k}(E)$? 

Our next result deals with this question and answers it in the affirmative sense, replacing the capacity $\Gamma_{\hat k}$, $1\le k\le d$, by the capacity 
$\Gamma_{\hat k,+}$, which is a version of $\Gamma_{\hat k}$ in the sense that one replaces the real distributions in \eqref{Gammak} by positive measures. 
It is defined as follows, given a compact set $E\subset\Rn$, 
\begin{equation*}
\Gamma_{\hat k,+}(E)=\sup\mu(E),
\end{equation*}
the supremum taken over those positive measures $\mu$ supported on $E$ such
that the potentials $\mu*x_i/|x|^2$, $1\le i\le d$, $i\neq k$, are in the unit ball of $L^\infty(\Rn)$.

\begin{teo}\label{comparability2}
There exists some positive constant $C$ such that for all compact sets $E\subset\Rn$ 
$$C^{-1}\Gamma_{\hat k,+}(E)\le\Gamma(E)\le C\Gamma_{\hat k,+}(E).$$
\end{teo}

The paper is organized as follows, Section \ref{sectionhuov} contains some examples of capacities that are not comparable to analytic capacity. 
In Section \ref{sectionsketch} we present a sketch
of the proof of Theorem \ref{main}. It becomes clear that the proof depends
on two facts: the close relationship between the quantities one
obtains after symmetrization of the kernels~$1/z$
and~$x_i^{2n-1}/|x|^{2n}$, $i=1,2$, and a localization $L^\infty$ estimate for the
scalar kernels $x_i^{2n-1}/|x|^{2n}$. In Section \ref{sectionpermu} we deal with the
symmetrization issue and in Section \ref{sectionlocal} with the localization
estimate. In Section \ref{sectioncont} we show an exterior regularity property of $\gamma_n$ needed for the proof of Theorem \ref{main}. 
In Section \ref{sectionbigpiece} we prove Theorem \ref{comparability2}. Finally,  in Section \ref{sectionmiscellaneous} we present various additional results. 

\section{Preliminaries}

\subsection{Some capacities that are not comparable to analytic capacity}\label{sectionhuov}

Let $K$ be some Calder\'on-Zygmund kernel of homogeneity $-1$ and consider its associated capacity $\gamma_K$ which is defined as follows: for a compact set $E\subset\Rd$,
$$\gamma_{K}(E)=\sup\{|\langle T,1\rangle |\},$$
the supremum taken over all distributions supported on the set $E$ and such that $K*T$ is an $L^\infty-$ function with $\|K*T\|_\infty\le 1$. 

As we already stated in the Introduction, we are interested in characterizing which are the homogeneous Calder\'on-Zygmund kernels whose related 
capacity is comparable to the analytic capacity $\gamma$. We are as well interested in the open problem of fully characterizing the homogeneous Calder\'on-Zygmund operators whose 
boundedness in $L^2({\mathcal{H}^1|E})$ implies the rectifiability of $E$ (see \cite{mmv}, \cite{Leger} and \cite{cmpt}). 
We think that both characterizations are deep problems in the area as even the candidate classe of ``reasonable''
kernels for the problems is far from clear. The relation between the two problems is illustrated in the Proposition \ref{relation} below. As a consequence, Corollary 
\ref{noncomp} shows that for some Calder\'on-Zygmund kernels, 
the capacities related to them are not comparable to analytic capacity. 

\begin{prop}\label{relation}
 Let $E\subset\Rd$ be a compact set with ${\mathcal H}^1(E)<\infty$. Let $K$ be some Calder\'on-Zygmund kernel of homogeneity $-1$ and 
 $S_K$ its associated Calder\'on-Zygmund operator. If $\gamma_K(E)\approx\gamma(E)$ and $S_K:L^2({\mathcal H}^1|E)\rightarrow L^2({\mathcal H}^1|E)$, then $E$ is not 
purely unrectifiable.
\end{prop}
\begin{proof}
 Let $F\subset E$ be such that ${\mathcal H}^1(F)>0$ and ${\mathcal H}^1|F$ has linear growth. Set $\mu={\mathcal H}^1|F$. 
From the $L^2(\mu)-$boundedness of $S_K$, we get that each $S_K$ is of weak type $(1,1)$ with respect to $\mu$. This follows from the standard 
Calder\'on-Zygmund theory if the measure is doubling and by an argument from \cite{ntv2} in the general case. By a standard dualization process 
(see \cite{do}, \cite[Theorem 23]{christ}, \cite{uy} and \cite{mp}) we get that for each compact set $G\subset F$ with $0<\mu(G)<\infty$, there exists a function 
$h$ supported on $G$, $0\le h\le 1$, such that 
$\displaystyle{\int_G hd\mu\ge C\mu(G)}$ and $\|S_K(hd\mu)\|_\infty=\|K*hd\mu\|_\infty\le 1$. Therefore $\gamma_K(E)>0$ and $\gamma(E)>0$ as well. 
Then by \cite{davidvitushkin}, $E$ is not purely unrectifiable (recall that a set $E$ is purely unrectifiable if the intersection of E with any curve of finite length has zero 1-dimensional Hausdorff 
measure).
\end{proof}

From Proposition \ref{relation} we obtain the following corollary :

\begin{co}
Let $K$ be some Calder\'on-Zygmund kernel of homogeneity $-1$ and 
 $S_K$ its associated Calder\'on-Zygmund operator. Suppose $\gamma_K\approx\gamma$. If $E\subset\Rd$ is a compact set with ${\mathcal H}^1(E)<\infty$ and $S_K$ is bounded in 
$L^2({\mathcal H}^1|E)$, then $E$ is rectifiable.
\end{co}

\begin{proof}
If $E$ were not rectifiable, then taking a purely unrectifiable compact subset $F\subset E$ with ${\mathcal H}^1(F)>0$ and using that, by 
Proposition \ref{relation}, $\gamma_K(F)\approx\gamma(F)$, we would get that $F$ is not purely unrectifiable, a contradiction.
\end{proof}

In \cite{huovinen}, it is shown that there exist homogeneous kernels, such as $H(x_1,x_2)=\frac{x_1x_2^2}{|x|^4}$, $x=(x_1,x_2)\in\Rd$, 
whose corresponding singular integrals are $L^2$-bounded on purely unrectifiable sets. We consider now the capacity related to this kernel $H$, namely $\gamma_H$.
As a consequence of Proposition \ref{relation} we obtain the following corollary

\begin{co}\label{noncomp}
 There exists some compact set $E\subset\Rd$ with $\gamma(E)=0$ and $\gamma_{H}(E)>0$.
\end{co}

It is worth saying that Huovinen's method does not work for the kernels we are considering in \eqref{ourkernel}, namely his construction does not give a purely 
unrectifiable set when changing the kernel $H$ by the kernels in \eqref{ourkernel}.\newline

\subsection{Sketch of the proof of Theorem \ref{main}}\label{sectionsketch}

In this section we will sketch the proof of the two inequalities appearing in the statement of Theorem \ref{main}. 
The first one is the following, for a compact set $E\subset\Rd$,
\begin{equation}\label{firstineq}
\gamma_n(E) \le C \,\gamma(E).
\end{equation}
For the proof of this inequality we need to introduce
the Cauchy transform with respect to an underlying positive Radon
measure $\mu$ satisfying the linear growth condition
\begin{equation}\label{lineargrowth}
\mu(B(x,r)) \le C\,r, \quad x \in \Rd, \quad r\geq 0.
\end{equation}
Given $\epsilon > 0$ we define the truncated Cauchy transform at
level $\epsilon$ as
\begin{equation}\label{Cauchyep}
C_{\epsilon}(f \,\mu)(z)=\int_{|w-z|>\epsilon} \frac{f(w)} {w-z}
\, d\mu(w), \quad z \in \Rd,
\end{equation}
for $f \in L^2(\mu)$. For a finite measure $\mu$, the growth condition on $\mu$ insures that
each $C_{\epsilon}$ is a bounded operator on $L^2(\mu)$ with
operator norm $ \|C_{\epsilon}\|_{L^2(\mu)}$ possibly depending on~$\epsilon$. We say that the Cauchy transform is bounded on
$L^2(\mu)$ when the truncated Cauchy
transforms are uniformly bounded on~$L^2(\mu)$. Call $L(E)$ the
set of positive Radon measures supported on $E$ which satisfy
\eqref{lineargrowth} with $C=1$ . One defines the capacities $\gamma_{\op}(E)$ and $\gamma_+(E)$ by
$$
\gamma_{\op}(E) = \sup \{\mu(E): \mu \in L(E) \quad\text{and}\quad
\|C\|_{L^2(\mu)} \le 1 \},
$$
$$
\gamma_+(E) = \sup \{\mu(E): \mu \in L(E) \quad\text{and}\quad \|\frac 1{z}*\mu\|_\infty\le 1\}.
$$
Clearly $\gamma_+(E)\le\gamma(E)$. The deep result in \cite{semiad} asserts that in fact $\gamma_+(E)$ is comparable to the anality capacity of $E$.
In \cite{tolsaindiana}, it was proved that the capacitiy $\gamma_+(E)$ is comparable to $\gamma_{\op}(E)$, that is, for some positive
constant $C$ one has
\begin{equation}\label{GammaopGamma+}
C^{-1}\, \gamma_{\op}(E) \le \gamma_+(E) \le C\,
\gamma_{\op}(E),
\end{equation}
for each compact set $E \subset\Rd.$ We remind the reader that the first inequality in
\eqref{GammaopGamma+} depends on a simple but ingenious duality
argument due to Davie and {\O}ksendal (see \cite[p.139]{do},
 \cite[Theorem 23, p.107]{christ} and \cite[Lemma 4.2]{verdera}).

From the first inequality in \eqref{GammaopGamma+}  we get that for some constant~$C$ and all compact sets~$E$,
$$
\gamma_{\op}(E) \le C \, \gamma(E).
$$

 To prove \eqref{firstineq} we will estimate $\gamma_n(E)$ by a constant times
$\gamma_{\op}(E)$. The natural way to perform that is to introduce
the capacity $\gamma_{n,\op}(E)$ and check the validity of the two estimates
\begin{equation}\label{GammanGammanop}
\gamma_n(E) \le C \,\gamma_{n,\op}(E)
\end{equation}
and
\begin{equation}\label{GammanopGammaop}
\gamma_{n,\op}(E) \le C \,\gamma_{\op}(E).
\end{equation}

To define $\gamma_{n,\op}$, first we introduce the truncated transform $S_{n,\ep}(f
\,\mu)(x)$ associated with the vectorial kernel $K=(K_1,K_2)$ with $\displaystyle{K_i(x)=x_i^{2n-1}/|x|^{2n}}$, $i=1,2$, as in~\eqref{Cauchyep}, but
 with the Cauchy kernel replaced by the vector valued  kernel $K$ just defined. We also set
$$
\|S_n\|_{L^2(\mu)} = \sup_{\epsilon > 0}
\|S_{n,\ep}\|_{L^2(\mu)} ,
$$
and
$$
\gamma_{n, \op}(E) = \sup \{\mu(E): \mu \in L(E)
\quad\text{and}\quad \|S_n\|_{L^2(\mu)} \le 1\,\}.
$$

One proves \eqref{GammanopGammaop} by checking that the
symmetrization of the Cauchy kernel is controlled by the
symmetrization of kernel $K$ (see Lemma \ref{lowerbound} and Corollary \ref{desigualtats}). 
In fact, we prove in Corollary \ref{l2bound} that for a positive measure $\mu$ having linear growth, the $L^2(\mu)$ boundedness of the Cauchy transform is 
equivalent to the $L^2(\mu)$ boundedness of the operators $S_n$. Therefore, the capacities $\gamma_{n, \op}(E)$ and $\gamma_{\op}(E)$ are comparable. 
Here the fact that we are dealing with kernels of
homogeneity $-1$ plays a key role, because, as it is shown by Farag in \cite{Farag},
they enjoy a special positivity property which is missing in
general. See Section \ref{sectionpermu} for complete details.

The proof of \eqref{GammanGammanop} depends on Tolsa's
proof of $\gamma(E) \le C\, \gamma_{\op}(E).$
One of the technical points that we need to prove in our setting is a localization
result for the potentials we deal with in this case, namely for the potentials associated with the kernels $K_i$, $i=1,2$. Specifically, in Section \ref{sectionlocal} we
prove that there exists a positive constant~$C$ such that, for
each compactly supported distribution~$T$ and for each coordinate
$i$, we have
\begin{equation}\label{localiz}
\left\|\frac{x_i^{2n-1}}{|x|^{2n}} * \varphi_Q T \right\|_\infty  \leq
C\left(\left\|\frac{x_i^{2n-1}}{|x|^{2n}} * T \right\|_\infty+G(T)\right)
\end{equation}
for each square $Q$ and each  $\varphi_Q \in{\mathcal C}^\infty_ 0(Q)$
satisfying $\|\varphi_Q\|_\infty\le 1$ and $\|\nabla\varphi_Q\|_\infty \le l(Q)^{-1}$. 
Here $G(T)$ is some constant related to the linear growth of $T$ (see Section \ref{sectionlocal} for a definition).

Once \eqref{localiz} is at our disposition, we claim that inequality \eqref{GammanGammanop} can be
proved by adapting the scheme of the proof of Theorems~1.1 in~\cite{semiad} and 7.1 in~\cite{semiad2}.  As Lemma \ref{extreg}
shows, the capacities $\gamma_n$, $n\in\N$, enjoy
the exterior regularity property. This is also true for the
capacities $\gamma_{n,+}$, defined by
$$
\gamma_{n,+}(E)=\sup\left\{\mu(E):\left\|\frac{x_j^{2n-1}}{|x|^{2n}}*\mu\right\|_\infty\leq 1, j=1,2\right\},
$$
just by the weak $\star$ compactness of the set of positive
measures with total variation not exceeding $1$. Therefore
we can approximate a general compact set $E$ by sets which are
finite unions of squares of the same side length in such a way that
the capacities $\gamma_n$ and $\gamma_{n,+}$ of the
approximating sets are as close as we wish to those of $E$.
As in \eqref{GammaopGamma+}, one has, using the Davie-{\O}ksendal
Lemma for several operators \cite[Lemma 4.2]{mp},
\begin{equation*}
C^{-1}\, \gamma_{n,\op}(E) \le \gamma_{n,+}(E) \le C\,
\gamma_{n,\op}(E).
\end{equation*}
Thus we can assume, without loss of generality,  that $E$ is a
finite union of squares of the same size. This will allow to implement
an induction argument on the size of certain rectangles. The first step involves rectangles of diameter
comparable to the side length of the squares whose union is $E$.

The starting point of the general inductive step in~\cite{semiad} and \cite{semiad2} consists in
the construction of a positive Radon measure $\mu$ supported on a
compact set $F$ which approximates $E$ in an appropriate sense. The set $F$ is defined as the
union of a special family of squares $\{Q_i\}_{i=1}^N$ that cover the
set $E$ and approximate $E$ at an appropriate intermediate scale.
One then sets
$$F=\bigcup_{i=1}^NQ_i.$$
The construction of the approximating set~$F$ implies that 
$\gamma_{n,+}(F)\leq C\,\gamma_{n,+}(E)$. This part of the proof extends without any obstruction to our case
because of  the positivity properties of the symmetrization of our
kernels (see Section \ref{sectionpermu}). To construct the measure $\mu$, observe that the definition of $\gamma_{n}(E)$
gives us a real distribution $S_0$ supported on $E$ such that
\begin{enumerate}
 \item ${\gamma_{n}(E) \le 2 |\langle S_0,1\rangle |.}$\newline
 \item  $\displaystyle{\left\|\frac{x_j^{2n-1}}{|x|^{2n}}* S_0\right\|_\infty\leq 1,}$ \quad $1\leq j\leq 2$.
 \end{enumerate}

Consider now functions $\varphi_i\in{\mathcal C}_0^{\infty}(2Q_i)$,
$0\leq\varphi_i\leq 1$, $\|\varphi_i\|_\infty\le 1$ and $\|\nabla \varphi_i\|_\infty\leq
l(Q_i)^{-1}$  and
$\sum_{i=1}^N\varphi_i=1$ on $\bigcup_iQ_i$. We define now
simultaneously the measure~$\mu$ and an auxiliary measure~$\nu$,
which should be viewed as a model for~$S_0$ adapted to the family of
squares $\{Q_i\}_{i=1}^N$.  For each square~$Q_i$ take a concentric
segment~$\Sigma_i$ of length a small fixed fraction of $\gamma_{n}(E \cap 2Q_i)$ and set
$$
\mu=\sum_{i=1}^N\Hu_{|\Sigma_i} \quad \mbox{ and }\quad
\nu=\sum_{i=1}^N\frac{\langle S_0,\varphi_i\rangle }{\Hu(\Sigma_i)}\Hu_{|\Sigma_i}.
$$
We have $d\nu=bd\mu$, with
$\displaystyle{b=\frac{\langle \varphi_i,\nu_0\rangle}{\Hu(\Sigma_i)}}$ on
$\Sigma_i$. At this point we need to show that our function~$b$ is
bounded, to apply later a suitable $T(b)$ Theorem. To estimate
$\|b\|_{\infty}$ we use the localization inequalities \eqref{localiz}. Thus,  $|\langle S_0 , \varphi_i\rangle |\leq C\,\gamma_{n}(2Q_i\cap E),$ for $1\leq i\leq N$.
It is now easy to see that $\gamma_n(E) \le C \,\mu(F) $:
\begin{equation*}
\gamma_n(E) \le 2\,|\langle S_0,1\rangle | = 2\left|\sum_{i=1}^N \langle S_0 ,
\varphi_i\rangle \right| \le C \sum_{i=1}^N  \gamma_n(2Q_i\cap E) =
C \, \mu(F).
\end{equation*}

Notice that the
construction of $F$ and $\mu$ gives readily that $\gamma_n(E)
\le C\, \mu(F),$  and $ \gamma_{n,+}(F) \le C\,
\gamma_{n,+}(E)$, which tells us that $F$ is not too small
but also not too big.
 However, one cannot expect the
operator $S_n$ to be bounded on~$L^2(\mu)$. One has to carefully look for a
compact subset $G$ of $F$ such that $\mu(F) \le C\,\mu(G)$, the
restriction $\mu_G$ of $\mu$ to $G$ has linear growth and $S_n$ is bounded on
$L^2(\mu_G)$ with dimensional constants. This completes the proof
because then
\begin{equation*}
\begin{split}
\gamma_{n}(E) &\le C\, \mu(F) \le C\, \mu(G) \le
C\,\gamma_{n,\op}(G)  \le C\,\gamma_{n,\op}(F) \\*[5pt]
& \le C\,\gamma_{n,+}(F) \le C\, \gamma_{n,+}(E) \le
C\,\gamma_{n,\op}(E) .
\end{split}
\end{equation*}
We do not insist in summarizing the intricate details, which can be
found in \cite{semiad} and \cite{semiad2}, of the definition of the
set $G$ and of the application of the $T(b)$ Theorem of \cite{ntv}.\newline

The second inequality in Theorem \ref{main} is 
\begin{equation}\label{secondineq}
\gamma(E)\le C \gamma_n(E).
\end{equation}
Since by \cite{semiad}, $\gamma(E)\approx\gamma_{\op}(E),$ and as we mentioned above we have
\begin{equation}\label{GammaopGammanop}
\gamma_{\op}(E)\le C\gamma_{n,\op}(E),
\end{equation}
we get that $\gamma(E)\le C\gamma_{\op}(E)\le C\gamma_{n,\op}(E).$
The duality arguments used to prove the first inequality in \eqref{GammaopGamma+} can also be used in our setting, therefore $\gamma_{n,\op}(E)\le C\gamma_{n,+}(E)$ holds. Finally, by definition, $\gamma_{n,+}(E)\le\gamma_n(E).$ This shows how \eqref{secondineq} in Theorem \ref{main} can be proved. 

\section{Symmetrization process and $L^2-$boundedness}\label{sectionpermu}

The symmetrization process for the Cauchy kernel introduced in
\cite{me} has been succesfully applied to many problems of analytic
capacity and $L^2$ boundedness of the Cauchy integral operator (see
\cite{mv}, \cite{mmv}, \cite{semiad}, and the book \cite{pa}, for example). In the recent paper \cite{cmpt}, the symmetrization method was also used to give the first non-trivial examples of operators not 
directly related to the Cauchy transform whose $L^2-$boundedness implies rectifiability. 

Given three distinct points in the plane, $z_1$, $z_2$ and
$z_3$, one finds out, by an elementary computation that
\begin{equation}\label{curvatura}
c(z_1,z_2,z_3)^2=\sum_{\sigma}\frac
1{(z_{\sigma(1)}-z_{\sigma(3)})\overline{(z_{\sigma(2)}-z_{\sigma(3)})}}
\end{equation}
where the sum is taken over the permutations of the set
$\{1,2,3\}$ and $c(z_1,z_2,z_3)$ is {\em Menger curvature}, that
is, the inverse of the radius of the circle through $z_1$, $z_2$ and~$z_3$. In particular \eqref{curvatura} shows that the sum on the
right hand side is a non-negative quantity.

In $\Rd$ and for $1\leq i\leq 2$ the quantity
\begin{equation*}
 p_i(z_1,z_2,z_3) = K_i(z_1-z_2)\,K_i(z_1-z_3) + K_i(z_2-z_1)\,K_i(z_2-z_3) + K_i(z_3-z_1)\,K_i(z_3-z_2),\end{equation*}
 is the obvious analogue of the right hand side of~\eqref{curvatura} for the kernel $K_i(x)=x_i^{2n-1}/|x|^{2n}$. 
In \cite{cmpt} it was shown that for any three distinct points $z_1,z_2,z_3\in\Rd$, the quantities $p_i(z_1,z_2,z_3)$, $1\le i\le 2$, are also non negative and they vanish 
if and only if the three points are colinear. 

The relationship between the quantity $p_i(z_1,z_2,z_3)$, $1\leq
i\leq 2$, and the $L^2$~estimates of the operator with
kernel~$x_i^{2n-1}/|x|^{2n}$ is as follows. Take a compactly supported positive Radon
measure~$\mu$ in $\Rd$ with linear growth. Given $\ep>0$ consider
the truncated transform $T^i_\ep(\mu)$ of $\mu$
associated with the kernel $K_i$, as in Section \ref{sectionsketch}. Then we
have (see in~\cite{mv} the argument for the Cauchy integral
operator)
\begin{equation*}
\left|\int|T^i_{\ep}(\mu)(x)|^2\,d\mu(x)-\frac 1
3p_{i,\ep}(\mu)\right|\leq C\|\mu\|,
\end{equation*}
$C$ being a positive constant depending only on $n$ and the linear growth constant of $\mu$, and
$$
p_{i,\ep}(\mu)=\underset{S_\ep}{\iiint}p_i(x,y,z)\,d\mu(x)\,d\mu(y)\,d\mu(z),
$$
with
$$
S_\ep=\{(x,y,z):|x-y|>\ep,\, |x-z|>\ep \text{ and }|y-z|>\ep\}.
$$

It is worth saying now that for $n=1$ and $i=1,2$, $\displaystyle{p_i(z_1,z_2,z_3)=\frac 1 2c(z_1,z_2,z_3)^2}$. For $n>1$, it is in general not true that $p_i(z_1,z_2,z_3)$, $i=1,2$, is comparable 
to Menger curvature $\displaystyle{c(z_1,z_2,z_3)^2}$.
The next two lemmas show that the sum of the above defined permutations, $\displaystyle{p_1(z_1,z_2,z_3)+p_2(z_1,z_2,z_3)}$ is comparable to Menger curvature, $\displaystyle{c(z_1,z_2,z_3)^2}$.

\begin{lemma}
\label{lowerbound}
There exists a constant $c_1=c_1(n)$, such that for all distinct points $z_1,z_2,z_3\in\Rd$,
$$p_1(z_1,z_2,z_3)+p_2(z_1,z_2,z_3) \geq c_1 c(z_1,z_2,z_3)^2.$$
\end{lemma}
\begin{proof} It suffices to prove the claim for any triple $(0,z,w), z\neq w \in \Rd \setminus\{0\}$. For any line $L$ denote by $\theta_V (L)$ and $\theta_H(L)$ the 
smallest angle that $L$ forms with the vertical and horizontal axes respectively. Then at least two of the angles,
$$\theta_V (L_{0,z}),\quad \theta_V(L_{0,w}),\quad \theta_V(L_{z,w})$$
or at least two of the angles
$$\theta_H(L_{0,z}),\quad \theta_H(L_{0,w}),\quad \theta_H(L_{z,w})$$
are greater or equal than $\pi/4$.
Without loss of generality we can assume that
\begin{equation}
\label{farfromvert}
\theta_V (L_{0,z})\geq \frac{\pi}{4} \quad \text{and} \quad \theta_V(L_{0,w}) \geq \frac{\pi}{4}.
\end{equation}
Now let $\theta= \theta_V (L_{z,w})$. Let $c$ be some very small constant, depending on $n$, that will be chosen later. \\

\textit{Case 1}: $\theta \geq c$.

As in Lemma 2.3 in \cite{cmpt}, we have that for $z=(x,y)$ and $w=(a,b)$,
\begin{equation}\label{dinouimig}
p_1(0,z,w) \geq n \left(\frac{|x|}{|z|}\right)^{2n-2}\left(\frac{|a|}{|w|}\right)^{2n-2}\left(\frac{|x-a|}{|z-w|}\right)^{2n-2} \frac{\sin ^2(z,w)}{|z-w|^2}.
\end{equation}
By (\ref{farfromvert}) we have that
\begin{equation}\label{comparab}
\frac{|x|}{|z|}> \frac{1}{2}, \qquad \frac{|a|}{|w|}> \frac{1}{2}
\end{equation}
and by the assumption in this case,
$$\frac{|x-a|}{|z-w|}\geq \sin c.$$
Furthermore,
$$c(0,z,w)=\dfrac{2 \sin(z,w)}{|z-w|}.$$
By \eqref{dinouimig},
$$p_1(0,z,w)\geq c_1 c(0,z,w)^2,$$
for some positive constant $c_1$ depending on $n$.\newline

\textit{Case 2}: $\theta < c$.

In this case, notice that by \eqref{comparab},
\begin{equation*}
\begin{split} 
\big||x|-|a|\big| &\leq |x-a| =|z-w| \sin \theta  \leq |z|\sin \theta + |w| \sin \theta \\
&\leq 2 |x| \sin \theta + 2 |a| \sin \theta.
\end{split}
\end{equation*}
Hence,
$$\frac{1-2 \sin \theta}{1+2 \sin \theta} |a| \leq |x| \leq \frac{1+2 \sin\theta}{1-2 \sin \theta}|a|$$
and since $\theta<c$ and $c$ will be chosen very small, it follows that
\begin{equation}
\label{xcompa}
\frac{|a|}{2} \leq |x| \leq 2 |a|.
\end{equation}

Combining (\ref{xcompa}) and (\ref{comparab}) we obtain that 
\begin{equation}
\label{zcompw}
\frac{|w|}{4} \leq |z| \leq 4 |w|.
\end{equation}
Expanding $p_1 (0,z,w)$ we get
\begin{equation*}
\begin{split}
p_1(0,z,w)&=\frac{x^{2n-1}a^{2n-1}}{|z|^{2n}|w|^{2n}}+\frac{(x-a)^{2n-1}}{|z-w|^{2n}}\left(\frac{x^{2n-1}}{|z|^{2n}}-\frac{a^{2n-1}}{|w|^{2n}} \right) \\
&=A+B,
\end{split}
\end{equation*}
where the last equality is a definition for $A$ and $B$.
Since
\begin{equation*}
|x^{2n-1}-a^{2n-1}|\leq |x-a|\left(|x|^{2n-2}+|x|^{2n-3}|a|+ \dots + |x||a|^{2n-3}+ |a|^{2n-2} \right),
\end{equation*}
then by (\ref{xcompa}),
\begin{equation}
\label{wev1}
|x^{2n-1}-a^{2n-1}| \leq (2n-1)2^{2n-2}|x-a||x|^{2n-2}.
\end{equation} 
Arguing in the same way and using (\ref{zcompw}) we obtain
\begin{equation}
\label{wev2}
\frac{\big||w|^{2n}-|z|^{2n}\big|}{|z|^{2n}|w|^{2n}} \leq 8n\;4^{2n-1} \frac{\big||z|-|w|\big|}{|w|^{2n+1}}.
\end{equation} 
Notice that
\begin{equation*}
\frac{x^{2n-1}}{|z|^{2n}}-\frac{a^{2n-1}}{|w|^{2n}}=\frac{x^{2n-1}-a^{2n-1}}{|z|^{2n}}+a^{2n-1}\left( \frac{1}{|z|^{2n}}-\frac{1}{|w|^{2n}} \right).
\end{equation*}
Therefore from (\ref{wev1}) and (\ref{wev2}) we get
\begin{equation*}
\begin{split}
|B|&\leq \frac{(\sin \theta )^{2n-1}}{|z-w|} \left(\frac{(2n-1)2^{2n-2}|x-a||x|^{2n-2}}{|z|^{2n}}+\frac{8n\;4^{2n-1}\big||z|-|w|\big||a|^{2n-1}}{|w|^{2n+1}} \right) \\
&\leq (\sin \theta )^{2n-1} \left(\frac{(2n-1)2^{2n-2}}{|z|^{2}}+\frac{8n\;4^{2n-1}}{|w|^{2}} \right)\leq (\sin \theta )^{2n-1}\left( \frac{16n4^{2n-1}}{|w|^2} \right).
\end{split}
\end{equation*}
On the other hand, by (\ref{comparab}) and \eqref{zcompw},
$$|A|=\left( \frac{|x|}{|z|}\right)^{2n-1}\left( \frac{|a|}{|w|}\right)^{2n-1}\frac{1}{|w||z|}\geq \left( \frac{1}{4} \right)^{2n} \frac{1}{|w|^2}$$
Therefore choosing $c\leq \dfrac{1}{10^4 n}$ we obtain that
$$p_1(0,z,w) \geq \frac{1}{2}\left( \frac{1}{4} \right)^{2n} \frac{1}{|w|^2}.$$
Since it follows easily that $c(0,z,w) \leq \dfrac{2}{|w|},$ the proof is complete.
\end{proof}

\begin{lemma}
There exists a positive constant $C=C(n)$ such that for all distinct points $z_1,z_2,z_3\in\Rd$,
$$p_i(z_1,z_2,z_3)\le C\; c(z_1,z_2,z_3)^2,\;\; 1\le i\le 2.$$
\end{lemma}
\begin{proof} Without loss of generality fix $i=1$.
Since $p_1$ is translation invariant, it is enough to estimate the permutations $p_1(0,z,w)$ for any two distinct points $z=(x,y)$, $w=(a,b)\in\Rd \setminus \{0\}$ such that
\begin{equation}
\label{trinv}
|z| \leq |z-w| \quad \text{and} \quad |w|\leq |z-w|.
\end{equation}  
As shown in Proposition 2.1 in \cite{cmpt},
\begin{equation}\label{idcmpt}
p_1(0,z,w)=\frac{A(z,w)}{|z|^{2n}|w|^{2n}|z-w|^{2n}},
\end{equation}
where \begin{equation}\label{estrella}A(z,w)=\sum_{k=1}^n\binom{n}{k}x^{2(n-k)}a^{2(n-k)}(x-a)^{2(n-k)}F_k(z,w)\end{equation}
and
$$F_k(z,w)=x^{2k-1}a^{2k-1}(y-b)^{2k}+x^{2k-1}(x-a)^{2k-1}b^{2k}-a^{2k-1}(x-a)^{2k-1}y^{2k}.$$
Notice also that
\begin{equation}
\label{extprod}
(xb-ay)^2=|z|^2|w|^2 \sin ^2(z,w)=\frac{1}{4}|z|^2|w|^2|z-w|^2 c(0,z,w)^2.
\end{equation}
\textit{Case 1}: $a=0$.

In this case, notice that $F_n(z,w)=x^{4n-2}b^{2n}$ and all sumands in \eqref{estrella} are zero, apart from the last one. Therefore, using 
(\ref{extprod}) and (\ref{trinv}),
\begin{equation*}
\begin{split}
p_1(0,z,w)&= \frac{x^{4n-4}b^{2n-2}}{|z|^{2n}|w|^{2n}|z-w|^{2n}} x^2 b^2\\
&=\frac{1}{4}\frac{|z|^2 |w|^2 |z-w|^2 x^{4n-4} b^{2n-2} }{|z|^{2n}|w|^{2n}|z-w|^{2n}}c(0,z,w)^2 \\
& \leq \frac{1}{4}\frac{|x|^{2n-2}}{|z-w|^{2n-2}}c(0,z,w)^2\leq \frac{1}{4}c(0,z,w)^2.
\end{split}
\end{equation*}
\textit{Case 2}: $a\neq 0$ and $b \neq 0$.

Let $t=x/a$ and $s=y/b$. Then $F_k$ can be rewritten as follows
\begin{equation*}
\begin{split}
\frac{F_k(z,w)}{a^{4k-2}b^{2k}}&=\left(\frac x a\right)^{2k-1}\left(\frac y b-1\right)^{2k}+\left(\frac x a\right)^{2k-1}\left(\frac x a-1\right)^{2k-1}-\left(\frac x a-1\right)^{2k-1}\left(\frac y b\right)^{2k}\\\\&=t^{2k-1}(s-1)^{2k}+t^{2k-1}(t-1)^{2k-1}-(t-1)^{2k-1}s^{2k}\\\\&=P(s,t),
\end{split}
\end{equation*}
the last identity being the definition of the polynomial $P(s,t)$. Then, for some polynomial $Q(s,t)$,
$$P(s,t)=(s-t)^2 Q(s,t),$$
because if we consider $P$ as a polynomial of the variable $s$ with parameter $t$ , i.e. $P_t(s):=P(s,t)$, we obtain easily that 
$$P_t(t)=P'_t(t)=0.$$

It is also immediate to check that the degree of $P$ is $4k-2$ and the smallest degree of the monomials of $P$ is $2k$. 

Therefore
$$Q(s,t)=\sum_{l+l'=2k-2}^{4k-4}c_{l,l'} t^l s^{l'}.$$
By  (\ref{extprod}) and (\ref{trinv}), for each $1\le k\le n$,
\begin{equation*}
\begin{split}
|F_k(z,w)|&=\left| a^{4k-4}b^{2k-2}(xb-ay)^2Q\left(\frac x a,\frac y b\right)\right|\\
&=\frac 1 4 |a|^{4k-4}|b|^{2k-2}|z|^2|w|^2|z-w|^2c(0,z,w)^2\;\left|Q\left(\frac x a,\frac y b\right)\right|\\
&\le C(n)|a|^{4k-4}|b|^{2k-2}\;|z|^2|w|^2|z-w|^2\;c(0,z,w)^2\sum_{l+l'=2k-2}^{4k-4} \left| \frac x a \right|^l \left| \frac y b \right|^{l'}\\
&= C(n)|z|^2|w|^2|z-w|^2 c(0,z,w)^2\sum_{l+l'=2k-2}^{4k-4} |a|^{4k-4-l}|b|^{2k-2-l'}|x|^l |y|^{l'}\\
&\le C(n)|z|^2|w|^2|z-w|^2 c(0,z,w)^2\sum_{l+l'=2k-2}^{4k-4} |w|^{6k-6-(l+l')}|z|^{l+l'}\\
&= C(n)|z|^{2k}|w|^{2k}|z-w|^{2}\;c(0,z,w)^2\sum_{l+l'=2k-2}^{4k-4}|w|^{4k-4-(l+l')}|z|^{l+l'-2k+2}\\
&\le C(n)|z|^{2k}|w|^{2k}|z-w|^{2}\;c(0,z,w)^2\sum_{l+l'=2k-2}^{4k-4}|z-w|^{4k-4-(l+l')+l+l'-2k+2}\\
&\le C(n)|z|^{2k}|w|^{2k}|z-w|^{2k}\;c(0,z,w)^2.
\end{split}
\end{equation*}

Then, from \eqref{idcmpt} we conclude that
\begin{equation*}
\begin{split}
p_1(0,z,w)&=\sum_{k=1}^n\binom{n}{k}\frac{x^{2(n-k)}a^{2(n-k)}(x-a)^{2(n-k)}}{|z|^{2n}|w|^{2n}|z-w|^{2n}}F_k(z,w)
\\\\&\le C(n)\left(\frac{|x|}{|z|}\right)^{2n}\left(\frac{|a|}{|w|}\right)^{2n}\left(\frac{|x-a|}{|z-w|}\right)^{2n}\;c(0,z,w)^2
\sum_{k=1}^n\binom{n}{k}\\\\&\le C(n)\;c(0,z,w)^2.
\end{split}
\end{equation*}

\textit{Case 3}: $b=0$.

In this case $F_k(z,w)=a^{2k-1}y^{2k}(x^{2k-1}-(x-a)^{2k-1})$. Hence by (\ref{extprod}) 
\begin{equation}
\label{f1}
F_1(z,w)=a^2y^2=\frac{1}{4}|z|^2 |w|^2|z-w|^2 c(0,z,w)^2.
\end{equation}
For $1<k\leq n$, by using \eqref{extprod} again,
\begin{equation*}
\begin{split}
F_k(z,w)&=a^{2k-1}y^{2k}(x^{2k-1}-(x-a)^{2k-1})\\
&=a^2 y^2 a^{2k-3} y^{2k-2}(x^{2k-1}-(x-a)^{2k-1})\\
&=\frac{1}{4}|z|^2|w|^2|z-w|^2 c(0,z,w)^2a^{2k-3}y^{2k-2}\sum_{j=0}^{2k-2}{{2k-1}\choose{j}}x^j a^{2k-1-j}.
\end{split}
\end{equation*}
And using (\ref{trinv}) we estimate,
\begin{equation*}
\begin{split}
|F_k(z,w)| &\leq C(n)c(0,z,w)^2|z|^2|w|^2|z-w|^2|w|^{2k-3}|z|^{2k-2} \sum_j^{2k-2} |z|^j |w|^{2k-1-j} \\
&\leq C(n)c(0,z,w)^2|z|^{2k}|w|^{2k}|z-w|^2  \sum_{j=0}^{2k-1} |z|^j |w|^{2k-2-j} \\
&= C(n)c(0,z,w)^2|z|^{2k}|w|^{2k}|z-w|^2 \sum_{j=0}^{2k-1} |z-w|^j |z-w|^{2k-2-j} \\
&\leq C(n)c(0,z,w)^2 |z|^{2k}|w|^{2k}|z-w|^{2k}.
\end{split}
\end{equation*}
The previous estimate combined with (\ref{f1}) implies that for $1\leq k \leq n$
$$|F_k(z,w)| \leq C(n)c(0,z,w)^2 |z|^{2k}|w|^{2k}|z-w|^{2k}.$$
Therefore, from \eqref{idcmpt} we derive that 
$$p_1(0,z,w)\leq C(n)\;c(0,z,w)^2$$
in an identical manner to case 1.

\end{proof}

From these two lemmas and the relationship between the symmetrization method and the $L^2$-norm we obtain the following:

\begin{co}\label{l2bound}
Let $S_n$ be the operator associated with the vectorial kernel $K=(K_1,K_2)$, with $K_i=x_i^{2n-1}/|x|^{2n}$, $n\in\N$ and $1\le i\le 2$.
If $\mu$  is a compactly supported positive measure in the plane
having linear growth, the Cauchy transform of $\mu$ is bounded on
$L^2(\mu)$ if and only if $S_n$ is bounded on $L^2(\mu)$.
\end{co}

We state now inequalities \eqref{GammanopGammaop} and  \eqref{GammaopGammanop}, because they are immediate consequences of the preceding corollary.

\begin{co}\label{desigualtats}
There exists a positive constant $C$ such that for any compact set $E\subset\Rd$, 
$$C^{-1}\;\gamma_{\op}(E)\le \gamma_{n, \op}(E)\leq C \,\gamma_{\op}(E).$$ 
\end{co}

It is worth to mention that for $n=1$,  it was proven in \cite{mpv} that corollary \ref{l2bound} remains valid if the operator $S_1$ is replaced by one of its coordinates, $S_1^1$ or $S_1^2$, (here $S_1^i$ is the operator with kernel $x_i/|x|^2$, $i=1,2$). 
\section{Growth conditions and localization}\label{sectionlocal}

We need the following reproduction formula for the kernels $K_i(x)=x_i^{2n-1}/|x|^{2n}$:

\begin{lemma}\label{reproductionformula}
If a function $f(x)$ has continuous derivatives up to order one, then it is representable in the form
\begin{equation}\label{repro}f(x)=(\varphi_1*K_1)(x)+(\varphi_2*K_2)(x),\, x\in\Rd,
\end{equation}
where for $i=1,2,$ \begin{equation}\label{tilda}\varphi_i= S_i(\partial_i f):=c\partial_if+\widetilde S_i(\partial_i f),\end{equation} for some constant $c$ and Calder\'on-Zygmund operators $\widetilde S_1$ and  $\widetilde S_2$. \end{lemma}

The proof of Lemma \ref{reproductionformula} is a consequence of the following two lemmas:

\begin{lemma}\label{coefficient}
For $m\geq 0$,
$$\sum_{k=0}^m\frac{(-1)^k\;2^{2k}\;k!}{(2k+1)!(m-k)!}=\frac 1{(2m+1)\;m!}.$$
\end{lemma}

\begin{proof}
We will show that
\begin{equation}\label{binomtransform}
\sum_{k=0}^m (-1)^k\;\binom{m}{k}\;a_k=\frac 1{2m+1},
\end{equation}
where
$$a_k=\frac{(2^k\;k!)^2}{(2k+1)!}.$$
Notice that \eqref{binomtransform} is equivalent to saying that the binomial transform of the sequence $a_k$ is $1/(2m+1)$ (see \cite{k}). Since the binomial transform is an involution of sequences, \eqref{binomtransform} is equivalent to regaining the original sequence $a_m$ by the inversion formula

\begin{equation}\label{sum}
a_m=\sum_{k=0}^m(-1)^k\;\binom{m}{k}\frac 1{2k+1}.
\end{equation}

To prove this identity, consider the Newton binomial formula
$$(1-x)^m=\sum_{k=0}^m(-1)^k\binom{m}{k}\;x^k$$
and multiply on both sides by $x^{-1/2}$. Integration between $0$ and $1$ gives now

$$\int_0^1 (1-x)^mx^{-1/2}dx=2\sum_{k=0}^m(-1)^k\;\binom{m}{k}\frac 1{2k+1}.$$

Recall that  $$\int_0^1 (1-x)^mx^{-1/2}dx=B\left(\frac 1 2,m+1\right),$$
$B(x,y)$ being the beta function.
Since it is easily seen that $$B\left(\frac 1 2,m+1\right)=2\frac{(2^m\;m!)^2}{(2m+1)!}=2a_m,$$
\eqref{sum} follows.
\end{proof}

The next lemma computes the Fourier transform of the kernel $K_i=x_i^{2n-1}/|x|^{2n}$, $1\le i\le 2$, $n\ge 1$, by using Lemma \ref{coefficient}.

\begin{lemma}\label{fourier}
For $n\geq 1$, $1\le i\le 2$, \begin{equation}\label{fouriercomputation}
\widehat{K_i}(\xi)=c \frac{\xi_i}{|\xi|^{2n}}p(\xi_1,\xi_2),
\end{equation}
where $p(\xi_1,\xi_2)$ is a homogeneous polynomial of degree $2n-2$ with no non-vanishing zeros. 
\end{lemma}

\begin{proof} Without loss of generality fix $i=1$.
For $n\ge 1$, let $E_n$ be the fundamental solution of the $n-$th power $\Delta^n$ of the Laplacean in the plane, that is
\begin{equation}
E_n(x)=|x|^{-(2-2n)}(\alpha+\beta\log|x|^2),
\end{equation}
for some positive constants $\alpha$ and $\beta$ depending on $n$ (see \cite{acl}).
Notice that, since $\Delta^n E_n=\delta_0$, then $$\widehat{(\partial_1^{2n-1}E_n)}(\xi)=c\frac{\xi_1^{2n-1}}{|\xi|^{2n}}$$for some constant $c$.
We will show that for some positive coefficients $b_{2m}$, $0\le m\le n-1$,
\begin{equation}\label{claim}
(\partial_1^{2n-1}E_n)(x)=c \frac{x_1}{|x|^{2n}}\sum_{m=0}^{n-1}b_{2m}x_1^{2m}x_2^{2(n-1-m)}.
\end{equation}
Notice that \eqref{fouriercomputation} follows from this fact.

To compute $\partial_1^{2n-1}E_n$, we will use the following formula from \cite{lz}:

\begin{equation}\label{derivative}
L(\partial)E_n=\sum_{\nu=0}^{n-1}\frac 1{2^\nu\;\nu!}\;\Delta^\nu L(x)\left(\frac 1 r\frac{\partial}{\partial r}\right)^{2n-1-\nu}E_n(r),
\end{equation}
where $r=|x|$ and $L(x)=x_1^{2n-1}$.  First notice that for $0\le\nu\le n-1$, we have

$$\Delta^\nu(x_1^{2n-1})=\binom{2n-1}{2\nu}(2\nu)!\;x_1^{2n-2\nu-1},$$
and for $0\le k\le n-1$,  one can check

$$\left(\frac 1 r\frac{\partial}{\partial r}\right)^{n+k}E_n(r)=2^n(n-1)!\;\frac{(-1)^k\;2^k\;k!}{r^{2+2k}}.$$

Plugging these, with $k=n-1-\nu$, into equation \eqref{derivative} we get 

\begin{equation}\label{sum2}
\partial_1^{2n-1}E_n(x)=2^{2n-1}(n-1)!\;\frac{x_1}{r^{2n}}\;\sum_{\nu=0}^{n-1}\;a_\nu \;x_1^{2(n-\nu-1)}\;r^{2\nu},
\end{equation}
where
$$a_\nu=\frac{(2\nu)!}{2^{2\nu}\;\nu!}\binom{2n-1}{2\nu}(-1)^{n-\nu-1}\;(n-1-\nu)!.$$

We claim that the homogeneous polinomial of degree $2n-2$ appearing in \eqref{sum2},

\begin{equation}\label{polynomial}
p(x_1,x_2)=\sum_{\nu=0}^{n-1}\;a_\nu \;x_1^{2(n-\nu-1)}\;r^{2\nu},
\end{equation}
has positive coefficients.  To prove this, write $r^2=x_1^2+x_2^2$. Then 
\begin{equation*}
\begin{split}
p(x)&=\sum_{\nu=0}^{n-1}\;a_\nu \;x_1^{2(n-\nu-1)}\;(x_1^2+x_2^2)^{\nu}\\&=\sum_{\nu=0}^{n-1}\sum_{k=0}^{\nu}a_\nu\binom{\nu}{k}x_1^{2(n-\nu+k-1)}x_2^{2(\nu-k)}=\sum_{m=0}^{n-1}b_{2m}x_1^{2m}x_2^{2(n-1-m)},
\end{split}
\end{equation*}
where for $0\le m\le n-1$, 
\begin{equation*}
\begin{split}
b_{2m}&=\sum_{k=1}^{m+1}a_{n-k}\binom{n-k}{m+1-k}\\*[7pt]&=\frac{(2n-1)!}{2^{2n}\;(n-m-1)!}\sum_{k=0}^{m}\frac{(-1)^k\;2^{2k}\;k!}{(2k+1)!\;(m-k)!}.
\end{split}\vspace{.5cm}
\end{equation*}

\noindent Applying now Lemma \ref{coefficient}, we get that for $0\le m\le n-1$,
\begin{equation*}
b_{2m}=\frac{(2n-1)!}{2^{2n}\;(n-m-1)!}\frac 1{(2m+1)\;m!}>0,
\end{equation*}
which completes the proof of \eqref{claim} and the lemma.

\end{proof}

\begin{proof}[Proof of Lemma \ref{reproductionformula}.] By Lemma \ref{fourier}, taking the Fourier transform in \eqref{repro} is equivalent to

$$\widehat f(\xi)=\widehat \varphi_1(\xi)\frac{\xi_1}{|\xi|^2}\frac{p(\xi_1,\xi_2)}{|\xi|^{2n-2}}+\widehat \varphi_2(\xi)\frac{\xi_2}{|\xi|^2}\frac{p(\xi_2,\xi_1)}{|\xi|^{2n-2}},$$
where $p$ is some homogeneous polynomial of degree $2n-2$ with no non-vanishing zeros.

\noindent Define the operator $R_1$ associated with the kernel
$$\widehat r_1(\xi_1,\xi_2)=\frac{p(\xi_1,\xi_2)}{|\xi|^{2n-2}}.$$
One defines also $R_2$, associated with $r_2$, where $r_2$ is given by $\widehat r_2(\xi_1,\xi_2)=\widehat r_1(\xi_2,\xi_1).$
Since $p$ is a homogeneous polynomial of degree $2n-2$, it can be decomposed as $$p(\xi_1,\xi_2)=\sum_{j=0}^{n-1} p_{2j}(\xi_1,\xi_2)|\xi|^{2n-2-2j},$$ 
where $p_{2j}$ are homogeneous harmonic polynomials of degree $2j$ (see \cite[3.1.2 p.~69]{St}). Therefore, the operators $R_i$, $1\le i \le 2$, can be written in the form
\begin{equation}\label{formduandi}
R_if=af+\pv\frac{\Omega(x/|x|)}{|x|^2}*f,
\end{equation}
 for some constant $a$ and $\Omega\in\mathcal{C}^\infty(S^1)$ with zero average. Consequently, by \cite[Theorem 4.15, p.82]{duandikoetxea}, the operators
$R_i$, $1\le i\le 2$, are invertible and the inverse operators, say $S_i$, $1\le i\le 2$, have the same form, namely the operators $S_i$, associated with the kernels $s_i$, $1\le i\le 2$, defined by
$$\widehat{s_1}(\xi)=\frac{|\xi|^{2n-2}}{p(\xi_1,\xi_2)}\;\;\;\;\mbox{and }\;\;\;\;\widehat{s_2}(\xi)=\frac{|\xi|^{2n-2}}{p(\xi_2,\xi_1)},$$
can be written as in \eqref{formduandi}, too. Therefore, setting 
$$\varphi_i=S_i(\partial_if),$$
for $1\le i\le 2,$ finishes the proof of Lemma \ref{reproductionformula}.\end{proof}

Observe that for a compactly supported
distribution $T$ with bounded Cauchy potential 
\begin{equation*}
\begin{split}
|\langle T,\varphi_Q\rangle |&=\left |\left\langle T,\frac{1}{\pi
z}*\overline\partial\varphi_Q\right\rangle \right|=
\left|\left\langle \frac{1}{\pi z}*T, \overline\partial\varphi_Q
\right\rangle\right|\\*[7pt] &\leq \frac 1 \pi\left\|\frac 1{z}*T
\right\|_\infty \,\|\overline\partial\varphi_Q\|_{L^1(Q)}\leq \,
\frac 1 \pi\left\|\frac 1{z}*T \right\|_\infty \,l(Q),
\end{split}
\end{equation*}
whenever $\varphi_Q$ satisfies
$\|\overline\partial\varphi_Q\|_{L^1(Q)} \le l(Q). $
\vspace{.8cm}

In our present case we do have a similar growth condition: if $T$ is a compactly supported distribution with bounded  potentials $K_1*T$ and $K_2*T$, then by Lemma \ref{reproductionformula}
\begin{equation}\label{ourgrowth}
\begin{split}
|\langle T,\varphi_Q\rangle |&=\left |\left\langle T,K_1*S_1(\partial_1\varphi_Q)+K_2*S_2(\partial_2\varphi_Q)\right\rangle \right|\\*[7pt] &\le
\left|\left\langle K_1*T, S_1(\partial_1\varphi_Q)
\right\rangle\right|+\left|\left\langle K_2*T, S_2(\partial_2\varphi_Q)
\right\rangle\right|\\*[7pt] &\leq  \left\|K_1*T
\right\|_\infty \,\|S_1(\partial_1\varphi_Q)\|_{L^1(\Rd)}+\left\|K_2*T
\right\|_\infty \,\|S_2(\partial_2\varphi_Q)\|_{L^1(\Rd)}\\*[7pt] &\leq\,
\left(\left\|K_1*T\right\|_\infty+\left\|K_2*T\right\|_\infty\right) \,l(Q),
\end{split}
\end{equation}
whenever $\varphi_Q$ satisfies
\begin{equation}\label{normalization1}
\|S_i(\partial_i\varphi_Q)\|_{L^1(\Rd)} \le l(Q),\,\,\,\mbox{for }\, i=1,2. 
\end{equation}

The next lemma states a sufficient condition for a test function to satisfy conditions \eqref{normalization1}.

\begin{lemma}\label{sublemma}
Let $1<q_0<\infty$ and assume that $f_Q$ is a test function supported on the square $Q$
satisfying, 
\begin{equation*}
\|\partial_i f_Q\|_{L^{q_0}(Q)}\le l(Q)^{2/q_0-1},\,\,\mbox{  for  } 1\le i\le 2.
\end{equation*}
Then, $$\|S_i(\partial_i f_Q)\|_{L^1(\Rd)}\le Cl(Q)\;\;\;\mbox{ for
}1\le i\le 2 .$$

\end{lemma}

\begin{proof}Without loss of generality fix $i=1$. Let $p_0$ be the dual exponent to $q_0$. By H\"older's inequality and the fact that the operator $S_1$ is bounded in $L^{q_0}(\Rd)$, $1<q_0<\infty$, we get
\begin{equation*}
\begin{split}
\|S_1(\partial_1 f_Q)\|_{L^1(2Q)}&\le Cl(Q)^{2/p_0}\|S_1(\partial_1 f_Q)\|_{L^{q_0}(\Rd)}\\&\le Cl(Q)^{2/p_0}\|\partial_1 f_Q\|_{L^{q_0}(Q)}\\&\le Cl(Q).
\end{split}
\end{equation*}

To estimate the $L^1$ norm outside $2Q$, notice first that since $\partial_1f_Q$ is supported on $Q$,  by \eqref{tilda}, $$\|S_1(\partial_1 f_Q)\|_{L^1((2Q)^c)}=\|\widetilde S_1(\partial_1 f_Q)\|_{L^1((2Q)^c)}.$$ Integrating by parts to take one derivative to the kernel $K$ of $\widetilde S_1$ and then using Fubini we obtain

\begin{equation*}
\begin{split}
\|S_1(\partial_1 f_Q)\|_{L^1((2Q)^c)}&=C
\,\int_{(2Q)^c}|\int_Q\partial_1
f_Q(z)K(z-y)\,dz|\,dy\\&=C
\,\int_{(2Q)^c}|\int_Q
f_Q(z)\partial_1K(z-y)\,dz|\,dy
\\& \le C \|f_Q\|_{L^1(Q)}\,l(Q)^{-1}\\&\le C\int |\nabla f_Q|,
\end{split}
\end{equation*}
the last estimate coming from the Cauchy-Schwarz  inequality, together with a well known result of Maz'ya (see \cite[1.1.4, p. 15]{mazya} and
\cite[1.2.2, p. 24]{mazya}) stating that
\begin{equation*}
\| f_Q\|_2\leq C\int|\nabla f_Q|.
\end{equation*}
Now H\"older's inequality  together with $\|\partial_i f_Q\|_{L^{q_0}(Q)}\le l(Q)^{2/q_0-1}$, $1\le i\le 2$, gives the desired estimate, namely $$\|S_1(\partial_1 f_Q)\|_{L^1((2Q)^c)}\le Cl(Q).$$ \end{proof}

Fix $1<q_0<2$. We say that a distribution $T$ has linear growth if
\begin{equation*}
G(T) = \sup_{\varphi_Q} \frac{|\langle T,\varphi_Q\rangle|}{l(Q)} <
\infty ,
\end{equation*}
where the supremum is taken over all $\varphi_Q \in \cc^\infty_0(Q)$
satisfying the normalization inequalities 
\begin{equation}\label{normalization}
\|\partial_i\varphi_Q\|_{L^{q_0}(Q)}\le 1,\;\;\mbox{ for }1\le i\le 2.
\end{equation}

Notice that from \eqref{ourgrowth} and Lemma \ref{sublemma}, if $T$ is a compactly supported distribution with bounded potentials 
$k_1*T$ and $k_2*T$, then $T$ has linear growth.\newline

We now state the localization lemma we need. 

\begin{lemma}\label{localization1}
Let $T$ be a compactly supported distribution in $\Rd$ 
such that $(x_i ^{2n-1}/ |x|^{2n}) *T$ is in $L^\infty(\Rd)$ for some $n\in\mathbb{N}$ and $1\leq i \leq 2$.
Let $Q$ be a square and assume that $\varphi_Q \in \cc^\infty_0(Q)$
satisfies $\|\varphi_Q\|_{\infty}\le C$ and $\|\nabla \varphi_Q\|_\infty \le l(Q)^{-1}$.
 Then $(x_i^{2n-1} / |x|^{2n}) * \varphi_Q T$ is in $L^\infty(\Rd)$ and
$$
\left\|\frac{x_i^{2n-1}}{|x|^{2n}}*\varphi_Q T\right\|_\infty\leq C\left(\left\| \frac{x_i^{2n-1}}{|x|^{2n}}*T\right\|_\infty+G(T)\right),
$$
for some positive constant $C$.
\end{lemma}

For the proof we need the following result.

\begin{lemma}\label{prelocalization}
 Let $T$ be a compactly supported distribution in $\Rd$ with linear growth and assume that
 $Q$ is a square and $\varphi_Q \in \cc^\infty_0(Q)$
satisfies $\|\varphi_Q\|_\infty\le 1$ and $\|\nabla \varphi_Q\|_{\infty} \le l(Q)^{-1}$. Then, for each coordinate $i$, the distribution $(x_i^{2n-1} / |x|^{2n})
* \varphi_Q T$ is an integrable function in the interior of $\frac 1 4 Q$ and 
$$
\int_{\frac 1 4 Q}\left|\left(\frac{x_i^{2n-1}}{|x|^{2n}}*\varphi_QT\right)(y)\right|dy\leq C \, G(T)\;l(Q)^2,
$$
where  $C$ is a positive constant.
\end{lemma}

\begin{proof}[Proof of Lemma \ref{prelocalization}]
The proof of this lemma follows the lines of Lemma 13 in \cite{mpv}, although now the growth conditions we have are different 
from the ones in \cite{mpv} (see \eqref{normalization}). We write the proof for the sake of completeness. 

Without loss of generality set $i=1$ and write
$K_1(x)=x_1^{2n-1}/|x|^{2n}$. We will 
prove that $K_1* \varphi_Q T$~is in
$L^{p_0}(2Q)$ where $p_0$ is the dual
exponent of $q_0$ (see \eqref{normalization}). Therefore we need to estimate the action of
$K_1 * \varphi_Q T$ on functions $\psi \in \cc^\infty_0(2Q)$ in
terms of $\|\psi\|_{q_0} $. We clearly have
$$
\langle K_1 * \varphi_Q T, \psi\rangle = \langle T, \varphi_Q(K_1 * \psi)\rangle.
$$
We claim that, for an appropriate
positive constant $C $, the test function
\begin{equation}\label{testf}
\frac{\varphi_Q(K_1 * \psi)}{C \,l(Q)^{\frac{2}{p_0}-1} \|\psi\|_{q_0}}
\end{equation}
satisfies the normalization inequalities \eqref{normalization} in
the definition of $G(T)$.  Once this is proved, by the definition of $G(T)$ we get that $|\langle K_1 * \varphi_Q T, \psi\rangle | \le C\,
l(Q)^{\frac{2}{p_0}}\|\psi\|_{q_0} \,G(T),$
and therefore $\|K_1 * \varphi_Q T \|_{L^{p_0}(2Q)} \le C\,
l(Q)^{\frac{2}{p_0}}G(T).$ Hence
\begin{equation*}
\begin{split}
\frac{1}{|\frac{1}{4}Q|}\int_{\frac{1}{4} Q} |(K_1 * \varphi_Q
T)(x)|\,dx &\le 16\frac{1}{|Q|}\int_Q |(K_1 * \varphi_Q
T)(x)|\,dx \\*[7pt]
& \le 16\left(\frac{1}{|Q|}\int_Q |(K_1
* \varphi_Q T)(x)|^{p_0} \,dx\right)^{\frac{1}{p_0}}\\*[7pt]
& \le C\,G(T),
\end{split}
\end{equation*}
which proves Lemma \ref{prelocalization}.

By Lemma \ref{sublemma}, to prove the claim we only have to show that for $1\le i\le 2$,
\begin{equation*}
\|\partial_i \left(\varphi_Q\,(K_1 * \psi)\right)\|_{L^{q_0}(Q)} \le
C\,\|\psi\|_{q_0} .
\end{equation*}
Clearly, for $1\le i\le 2$, we have
\begin{equation*}
\begin{split}
\partial_i \left(\varphi_Q \,(K_1 * \psi)\right) & = \varphi_Q\,\partial_i(K_1*\psi) + \partial_i\varphi_Q \,(K_1*\psi)= A+B,
\end{split}
\end{equation*}
where the last identity is the definition of $A$ and $B$.

To estimate the $L^{q_0}$-norm of $B$ we recall that 
$|K_1(x)| \le C\, |x|^{-1} .$ Hence, for $1\le i\le 2$,
\begin{equation*}
\begin{split}
\|\partial_i\varphi_Q \, (K_1*\psi)\|_{L^{q_0}(Q)}&\le
 C\,
 \|\partial_i\varphi_Q\|_{\infty}\left(\int_Q\left(\int_{2Q}\frac{|\psi(y)|}{|x-y|}dy\right)^{q_0}dx\right)^{1/q_0}\le C\|\psi\|_{q_0},
\end{split}
\end{equation*}
where the last inequality comes from Schur's Lemma applied to the operator with kernel $K(x,y)=|x-y|^{-1}\chi_{2Q}(x)\chi_{2Q}(y)$ and the fact that $\|\partial_i
\varphi_Q\|_{\infty} \le l(Q)^{-1}$, $1\le i\le 2$ .
We therefore  conclude that $\|B\|_{q_0}\le C\|\partial_i\varphi_Q\,(K_1*\psi)\|_{q_0}\le C\,\|\psi\|_{q_0}.$

We turn now to the term $A$.  
We remark that, for $1\le i\le 2$,
\begin{equation}\label{cz}
\partial_iK_1*\psi = c\,\psi + S(\psi),
\end{equation}
where $S$ is a smooth homogeneous convolution Calder\'{o}n-Zygmund
operator and $c$ some constant. This can be seen by
computing the Fourier transform of $\partial_iK_1$ and then using
that each homogeneous polynomial can be decomposed in terms of
homogeneous harmonic polynomials of lower degrees (see \cite[3.1.2
p.~69]{St}). Since Calder\'{o}n-Zygmund operators are bounded in $L^{q_0}(\Rd)$,
$1 <q_0 < \infty$, and $\|\varphi_Q\|_\infty\le C$, we get that $
\|A\|_{q_0}  \le C\,\|\psi\|_{q_0}.$
This completes the estimate of term $A$ and the proof of  \eqref{testf}.
\end{proof}

\begin{proof}[Proof of Lemma \ref{localization1}]
Here we argue as in Lemma 12 in \cite{mpv}. We write the proof for the sake of completeness. Without loss of generality take $i=1$. Let $x \in
\Rd\setminus \frac 3 2 Q .$ Then $K_1(x-y)\,\varphi_Q(y)$ is in
$\cc^\infty_0(Q)$ as a function of $y.$ Since for all $y\in\Rd$, $
|\partial_i (K_1(x-y)\,\varphi_Q(y))| \le C\, l(Q)^{-2},\,\,\,1\le i\le 2,
$ the function $c
\;l(Q)\,K_1(x-y)\,\varphi_Q(y)$ satisfies the normalization
conditions \eqref{normalization} for some small constant $c$. Therefore 
$$
|(K_1 * \varphi_Q T)(x)| = |\langle T, K_1(x-\cdot)\,\varphi_Q
\rangle| \le c^{-1}\,G(T),
$$
for all $x\in\Rd\setminus{\frac 3 2 Q}$. We are now left with the case  $x\in \frac 3 2 Q$.  Since $K_1*T$
and $\varphi_Q$ are bounded functions, we can write
$$
 |(K_1*\varphi_QT)(x)|\leq|(K_1*\varphi_QT)(x)-\varphi_Q(x)(K_1*T)(x)|+\|\varphi_Q\|_\infty\|K_1*T\|_\infty.
$$

Let $\psi_Q\in{\mathcal C}_0^{\infty}(\Rd)$ be such that $\psi_Q\equiv
1$ in $2Q$, $\psi_Q\equiv 0$ in $(4Q)^c$,
$\|\psi_Q\|_\infty\le C$ and $\|\nabla\psi_Q\|_\infty\leq C\,l(Q)^{-1}$. Then one is tempted to write
\begin{multline*}
 |(K_1*\varphi_QT)(x)-\varphi_Q(x)(K_1*T)(x)|\leq|\langle T,\psi_Q(\varphi_Q-\varphi_Q(x))K_1(x-\cdot)\rangle|\\*[5pt]
 +\|\varphi_Q \|_{\infty}|\langle T,(1-\psi_Q)K_1(x-\cdot)\rangle|.
\end{multline*}
The problem is that the first term on the right hand side above
does not make any sense because $T$ is acting on a function of $y$
which is not necessarily differentiable at the point $x$.  To
overcome this difficulty one needs to resort to a standard
regularization process. Take $\chi \in \cc^\infty(B(0,1))$ such
that $\int \chi = 1$ and set $\chi_\ep(x)=
\ep^{-2}\,\chi(x/\ep)$. It is enough to prove that $\chi_\ep*K_1*\varphi_QT$ is uniformly bounded, since $\chi_\ep*K_1*\varphi_QT$ converges weakly to $K_1*\varphi_QT$ in the distributinal sense, as $\ep\to 0$.
We have
\begin{equation*}
\begin{split}
|(\chi_\ep*K_1*\varphi_QT)(&x)-\varphi_Q(x)(\chi_\ep*K_1*T)(x)|\\*[7pt]
&\le
|\langle T,\psi_Q(\varphi_Q-\varphi_Q(x))
(\chi_\ep*K_1)(x-\cdot)\rangle|\\*[7pt]
&\quad+
\|\varphi_Q\|_{\infty}|\langle T,(1-\psi_Q)(\chi_\ep*K_1)(x-\cdot)\rangle|\\*[7pt]
&=A_1+A_2.
\end{split}
\end{equation*}
To deal with term $A_1$ set
$K_{1,\ep}^x(y)=(\chi_\ep*K_1)(x-y).$ We claim that, for an
appropriate small  constant $c$, the test function
$$f_Q=c\;l(Q)\psi_Q(\varphi_Q-\varphi_Q(x))K_{1,\ep}^{x},$$
satisfies the normalization inequalities \eqref{normalization} in
the definition of $G(T)$, with $\varphi_Q$ replaced by $f_Q$ and $Q$
by $4Q$. If this is the case, then
$$A_1\leq c^{-1} l(Q)^{-1}|\langle T,f_Q\rangle|\leq C\,G(T).$$

To prove the normalization inequalities \eqref{normalization} for
the function $f_Q$ we have to show that for $1\le i\le 2$, 
\begin{equation}\label{lq2}
\|\partial_i f_Q\|_{L^{q_0}(4\,Q)}\le Cl(Q)^{2/q_0-1}.
\end{equation}

To prove \eqref{lq2} we first notice that the regularized kernel
$\chi_\ep*K_1$ satisfies the inequality
\begin{equation}\label{regkernel}
|(\chi_\ep* \,K_1)(x)| \le \frac{C}{|x|}, \quad
x \in \Rd\setminus \{0\},
\end{equation}
where $C$ is a positive constant, which, in particular, is
independent of $\epsilon$. This can be proved by standard
estimates which we omit. Moreover, by \eqref{cz}, for $1\le i\le 2$, we have
\begin{equation*}
(\chi_\ep* \partial_i \,K_1)(x) = c\,\chi_\ep(x) + (\chi_\ep *
S)(x),
\end{equation*}
where $S$ is a smooth homogeneous convolution Calder\'{o}n-Zygmund
operator. As such, its kernel $H$ satisfies the usual growth
condition $|H(x)| \le C/ |x|^2$. From this is not difficult to
show that for some positive constant $C$,
\begin{equation}\label{regcz}
|(\chi_\ep * S)(x)| \le \frac{C}{|x|^{2}},\quad x \in \Rd\setminus
\{0\}.
\end{equation}
We have, for $1\le i\le 2$,
\begin{equation*}
\partial_i \left(\psi_Q(\varphi_Q -\varphi_Q(x))k_{\ep}^{1,x}\right)  =
\psi_Q \,(\varphi_Q -\varphi_Q(x)) \partial_i\,
k_\ep^{1,x}+
\partial_i(\psi_Q(\varphi_Q
-\varphi_Q(x)))\, k_\ep^{1,x}.
\end{equation*}
Therefore
\begin{equation*}
\begin{split}
\|\partial_i f_Q\|_{L^{q_0}(4Q)}&\leq Cl(Q)\left(\int_{4Q}|\psi_Q(y)
\,(\varphi_Q(y) -\varphi_Q(x))\, \partial_ik_{\ep}^{1,x}(y)|^{q_0}
\,dy\right)^{\frac 1 {q_0}}\\*[7pt] &\quad 
+Cl(Q)\left(\int_{4Q}|\partial_i
\left(\psi_Q(\varphi_Q -\varphi_Q(x)\right)\,k_{\ep}^{1,x}(y) |^{q_0}\,dy\right)^{\frac 1 {q_0}}\\*[7pt]&
 =A_{11}+A_{12}.
\end{split}
\end{equation*}
Using \eqref{regkernel} one obtains
\begin{equation*}
A_{12}\le C l(Q)\frac
1{l(Q)}\left(\int_{4Q}|(k_\ep^{1,x})(y)|^{q_0}\,dy\right)^{\frac 1 {q_0}}\leq Cl(Q)^{\frac 2 {q_0}-1}.
\end{equation*}
To estimate $A_{11}$ we resort to \eqref{regcz} and the fact that $q_0<2$, which yields
\begin{equation*}
\begin{split}
A_{11}&=Cl(Q)\left(\int_{4Q} |\psi_Q(y)(\varphi_Q(y)
-\varphi_Q(x))\partial_ik_{\ep}^{1,x}(y) |^{q_0} \,dy\right)^{\frac 1 {q_0}}\\*[7pt] &\leq
Cl(Q)\|\nabla\varphi_Q\|_\infty\left(\int_{4Q}
\frac{dy}{|y-x|^{q_0}}\, dy\right)^{\frac 1 {q_0}}\leq Cl(Q)^{\frac 2 {q_0}-1}.
\end{split}
\end{equation*}

We now turn to $A_2$. By Lemma \ref{prelocalization}, there exists a Lebesgue point of $K_1*\varphi_QT$, $x_0\in Q$, such that $|(K_1*\psi_QT)(x_0)|\leq C\, G(T)$.
 Then
$$|(K_1*(1-\psi_Q)T)(x_0)|\leq C\,(\|K_1*T\|_\infty +G(T)).$$
The analogous inequality holds as well for the regularized
potentials appearing in $A_2$, for $\epsilon$ small enough and with constants independent of $\epsilon$. Therefore
$$A_2\leq C\,|\langle T,(1-\psi_Q)(k_\ep^{1,x}-k_\ep^{1,x_0})\rangle|+C\,(\|K_1*T\|_\infty +G(T)).$$

To estimate $|\langle
T,(1-\psi_Q)(k_\ep^{1,x}-k_\ep^{1,x_0})\rangle|$, we decompose
$\Rd\setminus \{x\}$ into a union of rings $$N_j=\{z\in
\Rd:2^j\,l(Q)\leq|z-x|\leq 2^{j+1}\,l(Q)\},\quad j\in\mathbb{Z},$$
and consider functions $\varphi_j$ in ${\mathcal
C}^\infty_0(\Rd)$, with support contained in
$$N^*_j=\{z\in
\Rd:2^{j-1}\,l(Q)\leq|z-x|\leq 2^{j+2}\,l(Q)\},\quad
j\in\mathbb{Z},$$ such that $\|\varphi_j\|_\infty\le C$ and $\|\nabla\varphi_j\|_\infty\leq C
\,(2^j\,l(Q))^{-1}$, and $\sum_j\varphi_j=1$ on
$\Rd\setminus\{x\}$. Since $x\in\frac 3 2 Q$  the smallest ring
$N^*_j$ that intersects $(2Q)^c$ is $N^*_{-3}$. Therefore  we have
\begin{equation*}
\begin{split}
 |\langle T,(1-\psi_Q)(k_\ep^{1,x}-k_\ep^{1,x_0})\rangle|
 &=\left|\left\langle T,\sum_{j\geq -3}\varphi_j(1-\psi_Q)(k_\ep^{1,x}-k_\ep^{1,x_0})\right\rangle\right|\\*[7pt]
 &\leq\left|\left\langle T,\sum_{j\in I}\varphi_{j}(1-\psi_Q)(k_\ep^{1,x}-k_\ep^{1,x_0})\right\rangle \right|\\*[7pt]
&\quad+\sum_{j\in J}|\langle T,\varphi_{j}(k_\ep^{1,x}-k_\ep^{1,x_0})\rangle|,
\end{split}
\end{equation*}
where $I$ denotes the set of indices $j\geq -3$ such that the
support of $\varphi_j$ intersects $4Q$  and $J$ the remaining
indices, namely those $j \geq -3 $ such that $\varphi_j$ vanishes
on $4Q$. Notice that the cardinality of $I$ is bounded by a
positive constant.

Set
$$g =C\,l(Q)\sum_{j\in I}\varphi_j(1-\psi_Q)\,(k_\ep^{1,x}-k_\ep^{1,x_0}),$$
and for $j\in J$
$$g_j=C\,2^{2j}\,l(Q)\,\varphi_j\,(k_\ep^{1,x}-k_\ep^{1,x_0}).$$
We now show that the test functions $g$ and $g_j$, $j\in J$,
satisfy the normalization inequalities \eqref{normalization} in
the definition of $G(T)$ for an appropriate choice of the (small)
constant $C$ . Once this is available, using the linear growth
condition of $T$ we obtain
\begin{equation*}
\begin{split}
 |\langle T,(1-\psi_Q)(k_\ep^{1,x}-k_\ep^{1,x_0})\rangle |&\leq C l(Q)^{-1}|\langle T,g\rangle|\\*[7pt]
 &\quad + C \sum_{j\in J} (2^{2j}l(Q))^{-1}|\langle T,g_j\rangle |\\*[7pt]
 &\leq C\,G(T) + C\sum_{j\geq -3}2^{-j}\,G(T)\leq C\,G(T),
\end{split}
\end{equation*}
which completes the proof of Lemma \ref{localization1}.

Checking the normalization inequalities for $g$ and $g_j$ is easy.
First notice that the support of $g$ is contained in a square
$\lambda\,Q$ for some universal constant $\lambda$. On the other hand the support of $g_j$ is contained in $
2^{j+2}\,Q.$  By Lemma \ref{sublemma}, we have to show that for $1\le i\le 2$, and some $1<q_0<\infty$,
\begin{equation}\label{final1}
\|\partial_i g\|_{L^{q_0}(\lambda\, Q)}\le C
l(Q)^{2/q_0-1},\end{equation}
and for $j\in J$,
\begin{equation}\label{final2}
\|\partial_i g_j\|_{L^{q_0}(2^{j+2}\,Q)}\le C (2^j
l(Q))^{2/q_0-1}.
\end{equation}

To show \eqref{final1} we take $\partial_i$ in the definition of
$g$, apply Leibnitz's formula and estimate in the supremum norm each
term in the resulting sum. We get
$$
\|\partial_i g\|_\infty \le C\,l(Q) \sum_{k=0}^{1}
\frac{1}{l(Q)^{k}}\; \frac{1}{l(Q)^{2-k}} = C\,
\frac{1}{l(Q)},
$$
 which yields \eqref{final1} immediately.

For \eqref{final2}, applying a gradient estimate, we get
$$
\|\partial_i g_j\|_\infty \le C\,2^{2j}\,l(Q) \sum_{k=0}^{1}
\frac{1}{(2^j\,l(Q))^{k}}\; \frac{l(Q)}{(2^j\,l(Q))^{2+1-k}} =
C\, \frac{1}{2^j\,l(Q)},
$$
which yields \eqref{final2} readily.
\end{proof}

\section{Outer regularity}\label{sectioncont}

In what follows, we will show that the capacities $\gamma_n$ are outer regular. 
\begin{lemma}\label{extreg}
Let $\{E_k\}_k$ be a decreasing sequence of compact sets in $\Rd$, with intersection the compact set $E\subset\Rd$. 
Then $\displaystyle{\gamma_n(E)=\lim_{k\to\infty}\gamma_n(E_k).}$
\end{lemma}

\begin{proof}
The limit $\lim_{k\to\infty}\gamma^1_n(E_k)$ clearly exists and $\lim_{k\to\infty}\gamma_n(E_k)\geq \gamma_n(E)$. To prove the converse inequality,
let $T_k$ be a distribution supported on $E_k$ such that, for $1\le i\le 2$, $f_k^i=K_i*T_k$ is in the unit ball of $L^\infty(\Rd)$
and 
$$\gamma_n(E_k)-\frac1k < |\langle T_k,1\rangle|\leq \gamma_n(E_k).$$
By taking a subsequence if necessary, we may assume that, for $1\le i\le 2$, $f_k^i$ converges weakly $*$ in $L^\infty(\Rd)$ to some function 
$f^i$ such that $\|f^i\|_\infty\leq 1$.

We will show that $T_k$ converges to some distribution $T$ such that $T*K_1$ and $T*K_2$ are also in the unit ball of $L^\infty(\Rd)$.
Then 
$$\gamma_n(E)\geq \langle T,\,1\rangle = \lim_{k\to\infty}\langle T_k,\,1\rangle = 
\lim_{k\to\infty}\gamma_n(E_k),$$ and we will be done.

Let us first check that the limit of $\{T_k\}_k$ exists in the topology of distributions. 
This is equivalent to saying that, for any $\vphi\in\CC_c^\infty(\Rd)$,
the limit $\lim_{k\to\infty}\langle T_k,\vphi\rangle$ exists.
Using the reproducing formula \eqref{repro}, we deduce that
\begin{align*}
\langle T_k,\vphi\rangle & = \langle T_k,\,S_1(\partial_1 \vphi)*K_1+S_2(\partial_2 \vphi)*K_2\rangle = \langle T_k*K_1,\,S_1(\partial_1 \vphi)\rangle+\langle T_k*K_2,\,S_2(\partial_2 \vphi)\rangle,
\end{align*}
which is convergent, since by Lemma \ref{sublemma}, $S_i(\partial_i \vphi)\in L^1(\Rd)$, $1\le i\le 2$, and $f_k^i= T_k*K_1$, $1\le i\le 2,$ is weak * convergent in 
$L^\infty(\Rd)$.

To see that, for $1\le i\le 2$, $T*K_i$ is  in the unit ball of $L^\infty(\Rd)$, 
we take a radial function $\chi\in\CC^\infty(\Rd)$, $\int\chi=1$, supported in the unit ball and, as usual, we denote $\chi_\ve(x)
=\ve^{-2}\chi(\ve^{-1}x)$. Then 
it is enough to prove that 
$\chi_\ve*T*K_i$ is in the unit ball of $L^\infty(\Rd)$ for all $\ve>0$.
This follows easily: denoting $K_i^\ve = \chi_\ve*K_i$, for each $x\in\Rd$, we have
$$T_k* K_i^\ve(x) = \langle T_k,\,K_i^\ve(x-\cdot)\rangle.$$
Notice moreover that $\|T_k* K_i^\ve \|_\infty =\|\chi_\ve * (T_k*K_i)\|_\infty\le 1$.
Now, let $\psi_0\in\CC_c^\infty(\Rd)$ be such that it equals $1$ in the 1-neighborhood of $E$, so that $T_k = \psi_0\,T_k$ for $k$ big enough. Then
$$T_k* K_i^\ve(x) = \langle \psi_0\,T_k,\,K_i^\ve(x-\cdot)\rangle = \langle T_k,\,\psi_0\,K_i^\ve(x-\cdot)\rangle,$$
which converges to $\langle T,\,\psi_0\,K_i^\ve(x-\cdot)\rangle = \langle T,\,K_i^\ve(x-\cdot)\rangle = T* K_i^\ve(x)$
as $k\to\infty$. Since $|T_k* K_i^\ve(x)|\leq 1$ for all $k$, we deduce that $|T* K_i^\ve(x)|\leq 1$ as wished, too.
\end{proof}

\section{Proof of Theorem \ref{comparability2}}\label{sectionbigpiece}

For the proof of Theorem  \ref{comparability2}, recall the following result from \cite{mpv}:

\begin{teo}(\cite{mpv})
For a compact set $E\subset\Rn$,
\begin{equation}\label{mpvth}
\Gamma(E)\approx\sup\mu(E),
\end{equation}
the supremum taken over those positive measures $\mu$ supported on $E$ with linear growth such that for $1\le i\le d$, $i\neq k$, the potentials $\ds{\frac{x_i}{|x|^2}*\mu}$ 
are in $L^\infty(\mu)$ with $\ds{\Big\|\frac{x_i}{|x|^2}*\mu\Big\|_{\infty}\le 1}$.
\end{teo}
Notice that the only difference between \eqref{mpvth} and Theorem \ref{comparability2} is the extra linear growth condition required on the positive measure $\mu$. 
Hence, to prove Theorem \ref{comparability2}, we have to get rid of this growth condition and still mantain the comparability between the capacities. 
Below, in Lemma \ref{bigp}, we show that if we are given a positive measure supported on $E$ with  $\ds{\|\frac{x_i}{|x|^2}*\mu\|_{\infty}\le 1}$ for $i\ne k$, $1\le i\le d$, 
then this measure grows linearly in a big piece of its support $E$. Thus Theorem \ref{comparability2} holds.\newline

For a Borel measure $\mu$, the curvature of $\mu$, which was introduced in \cite{me}, is the nonnegative number $c^2(\mu)$ defined by
 
$$c^2(\mu)=\iiint c(x,y,z)^2d\mu(x)d\mu(y)d\mu(z),$$
where $c(x,y,z)$ is the inverse of the radius of the circumcircle of the triangle $(x,y,z)$, that is the Menger curvature of the triple 
$(x,y,z)$ (see Section \ref{sectionpermu}).

The following result, that will be needed in what follows, is a version of \cite[Lemma 5.2]{tolsaindiana} for $\Rn$. Its proof uses the curvature theorem of 
G. David and L\'eger \cite[Proposition 1.2]{Leger}.

\begin{lemma} \label{curvgran}
Let $\mu$ be some Radon measure supported on $B(x_0,R)$, with
$$\Theta^*_\mu(x)=\underset{r\to 0}{\limsup}\;\frac{\mu(B(x,r))}{r}\le 1\;\;\mbox{ for }\;\mu\;\mbox{-a.e. }x\in \Rn.$$
If $c^2(\mu)\leq C_2\,\mu(B(x_0,R))$, then $\mu(B(x_0,R)) \leq MR$,
where $M$ is some constant depending only on $C_2$.
\end{lemma}

From the preceding lemma we get the following.

\begin{lemma}\label{curvgran2}
Let $\mu$ be a finite Borel measure  supported on a ball $B(x_0,R)$.
Suppose that 
$$\Theta^1_\mu(x)=\lim_{r\to0}\frac{\mu(B(x,r))}{r}=0\qquad \mbox{for $\mu$-a.e.\ $x\in\Rn$}.$$
Then,
\begin{equation}\label{eq0}
\biggl(\frac{\mu(B(x_0,R))}R\biggr)^2 \leq c_1\,\frac{c^2(\mu)}{\mu(B(x_0,R))},
\end{equation}
for some absolute constant $c_1$.
\end{lemma}

\begin{proof}
Consider the measure $\wt\mu = \left(\frac{\|\mu\|}{c^2(\mu)}\right)^{1/2}\mu$.
Notice that
$$c^2(\wt\mu) = \biggl(\frac{\|\mu\|}{c^2(\mu)}\biggr)^{3/2}\,c^2(\mu) = \|\wt\mu\|.$$
Applying Lemma \ref{curvgran} to $\wt\mu$ with $C_2=1$, we infer that there exists an absolute
constant $M$ such that
$\wt\mu(B(x_0,R)) \leq MR$, and thus
$$\mu(B(x_0,R))\leq M
\left(\frac{c^2(\mu)}{\mu(B(x_0,R))}\right)^{1/2} R,$$
which is equivalent to \rf{eq0}, with $c_1=M^2$.
\end{proof}

\begin{remark}
For $x,y,z\in\Rn$ set $K_i(x)=x_i/|x|^2$, $1\le i\le d$, and let 
$$p_i(x,y,z)=K_i(x-y)\,K_i(x-z) + K_i(y-x)\,K_i(y-z) + K_i(z-x)\,K_i(z-y).$$

Given any subset of $d-1$ elements of $\{1,2,\cdots,d\}$, ${\mathcal S}_{d-1}$, we define, for a positive measure $\mu$ (without atoms, say), 
$$p(\mu) =\sum_{i\in{\mathcal S}_{d-1}}\iiint p_i(x,y,z)\,d\mu(x)\,d\mu(y)\,d\mu(z).$$

Due \cite[Corollary 2 and Theorem 4]{mpv}, Lemma \ref{curvgran} also holds in $\Rn$ when replacing the Menger curvature by the permutations 
associated with any set of $d-1$ components of the vectorial kernel $x/|x|^2$ in $\Rn$. Therefore we recover Lemma \ref{curvgran} and Lemma \ref{curvgran2} 
with $c^2(\mu)$ replaced by $p(\mu)$.
\end{remark}

Given $M>0$, we say that a ball $B=B(x,r)$ is non $M$-Ahlfors (or simply, a non Ahlfors ball) if 
$$\Theta_\mu(B):=\frac{\mu(B)}{r} >M.$$

\begin{lemma}\label{bigp}
Let $\mu$ be a positive measure on $\Rn$ such that $\Big\|\dfrac{x_i}{|x|^2}*\mu\Big\|_{\infty}\leq 1$, for $i\in{\mathcal S}_{d-1}$. 
Let $A_M^\mu\subset\Rn$ be the union of all non $M$-Ahlfors balls. If $M$ is big enough, then
$$\mu(A_M^\mu)\leq \frac12\,\mu(\Rn).$$
\end{lemma}

\begin{proof}
Let $\vphi$ be a non negative radial $\CC^\infty$ function supported on $B(0,1)$ with $L^1$ norm equal to $1$, and denote $\vphi_t(x)=t^{-n}\vphi(x/t)$, for $t>0$. Observe 
that for $i\in{\mathcal S}_{d-1}$, the measure $\mu_t= \vphi_t*\mu$
satisfies
$$\Big\|\dfrac{x_i}{|x|^2}*\mu_t\Big\|_{\infty} = \Big\|\vphi_t*(\dfrac{x_i}{|x|^2}*\mu)\Big\|_{\infty}\leq 1.$$
Moreover, $\Theta_{\mu_t}^1(x)=0$ for every $x\in\Rn$ and $\mu_t$ has linear growth with some constant
depending on $t$ (since the density of $\mu_t$ is a $\CC^\infty$ function with compact support), and thus, 
\begin{equation}\label{eqpmut0}
p(\mu_t) = 3\sum_{i\in{\mathcal S}_{d-1}}\,\Big\|\dfrac{x_i}{|x|^2}*\mu_t\Big\|_{L^2(\mu_t)}^2\leq 3\,\sum_{i\in{\mathcal S}_{d-1}}\Big\|\dfrac{x_i}{|x|^2}*\mu_t\Big\|_{\infty}^2\,\mu_t(\Rn)\leq 3(d-1)\,\|\mu\|.
\end{equation}

For $t>0$, denote
$$A_{M,t}^\mu =\bigcup_{\substack{B\text{ ball}:\,\Theta_\mu(B)\geq M\\
r(B)\geq t}} B.$$
Notice that if $r(B)\geq t$, then $\mu_t(2B)\geq \mu(B)$
and thus $\Theta_{\mu_t}(2B)\geq \Theta_\mu(B)/2$. Then by the preceding remark, if $B$ is one of the balls appearing in the union that defines $A_{M,t}^\mu$,
\begin{equation}\label{eqpmut}
p(\mu_t\rest 2B) \geq c_1^{-1}\,\frac{M^2}4\,\mu_t(2B)\geq c_1^{-1}\,\frac{M^2}4\,\mu(B).
\end{equation}

By the $5r$-covering lemma, there exists a family of non $M$-Ahlfors balls (for $\mu$), $B_j$, $j\in I$, such that the balls $2B_j$ are disjoint, and
$$A_{M,t}^\mu\subset \bigcup_{j\in I} 10B_j.$$
Moreover, the balls $B_j$ can be taken so that $aB_j$ is an $M$-Ahlfors ball for each $a\geq 2$ (just by considering maximal balls in the union that defines $A_{M,r}^\mu$). So we have
$$\mu(10B_j)\leq 10M\,r(B_j)\leq 10\,\mu(B_j).$$
Then, by \rf{eqpmut} and \rf{eqpmut0},
\begin{equation*}
\begin{split}
\mu(A_{M,t}^\mu )&\leq \sum_{j\in I} \mu(10B_j) \leq 10 \sum_{j\in I} \mu(B_j)
\leq \frac{40\,c_1}{M^2}\sum_{j\in I} p(\mu_t\rest 2B_j) 
\\&\leq \frac{40\,c_1}{M^2} p(\mu_t)\leq \frac{120(d-1)\,c_1}{M^2}\,\|\mu\|.
\end{split}
\end{equation*}
So if $M$ is chosen big enough, $\mu(A_{M,t}^\mu )\leq \mu(\Rn)/2$, and letting $t\to 0$, the lemma follows.
\end{proof}

\begin{remark}
Lemma \ref{curvgran2} and Lemma \ref{bigp} also hold in $\Rd$ replacing Menger curvature by the permutations of the kernel $x_i^{2n-1}/|x|^{2n}$, $1\le i\le 2$, $n\ge 1$, and the kernel $x_i/|x|^2$ by the kernel $x_i^{2n-1}/|x|^{2n}$,  respectively, because in \cite{cmpt} we proved David-L\'eger's theorem with these permutations instead of the usual curvature.\newline
\end{remark}


\section{Some remarks on related capacities}\label{sectionmiscellaneous}

\subsection{Extensions of Theorem \ref{main} to other capacities}

For $n\geq 1$ and $1\le j\le 2$, we set $K_j^n(x)=x_j^{2n-1}/|x|^{2n}$. For $n,\;m\ge 1$ and each compact set $E\subset\Rd$, we define the following capacity:
$$\gamma_{n,m}(E)=\sup\{|\langle T,1\rangle|\},$$
the supremum taken over all distributions $T$ supported on $E$ with potentials $T*K_1^n$ and $T*K_2^m$ in the unit ball of $L^\infty(\Rd)$.\newline

Using the same arguments as in Lemma \ref{reproductionformula}, one could show that
each function $f(x)$ with continuous derivatives up to order one is representable in the form $$f(x)=(\varphi_1*K_1^n)(x)+(\varphi_2*K_2^m)(x),\, x\in\Rd,$$
where the functions $\varphi_i$, $i=1,2$, are defined by the formula $\varphi_i(x)= S_i(\partial_i f)(x),$
with $S_i$, $1\le i\le 2$, being Calder\'on-Zygmund operators. Moreover the localization result of Lemma \ref{localization1} and the outer regularity property of 
Lemma \ref{extreg} also apply in this setting.
Therefore, using the same techniques, one obtains the comparability between analytic capacity and $\gamma_{n,m}$, namely that there exists some positive constant 
$C$ such that for all compact sets $E$ of the plane
$$C^{-1}\gamma_{n,m}(E)\le\gamma(E)\le C\gamma_{n,m}(E).$$\newline

In fact, following the proofs in \cite{semiad} and \cite{semiad2} (see also \cite{mpv}), one can show that for compact sets $E\subset\Rd$, a given capacity 
(associated twith some Calder\'on-Zygmund kernel $K$ with homogeneity $-1$) defined as
$$\gamma_K(E)=\sup\{|\langle T,1\rangle|:\;T\;\mbox{ distribution },\;\mbox{spt}\;T\subset E,\;\|T*K\|_\infty\le 1\},$$ is comparable to the analytic capacity $\gamma(E)$ provided the following properties hold:
\begin{itemize}
 \item The symmetrization method: one has to ensure that when symmetrizing the kernel $K$ (as in \eqref{curvatura}) the quantity obtained is non-negative and comparable to Menger curvature.
 \item The localization property: we need that our kernel $K$ localizes in the uniform norm. By this we mean that if $T$ is a compactly
supported distribution such that $T*K$ is a bounded function then $\varphi T*K$ is also
bounded for each compactly supported ${\mathcal C}^1$ function $\varphi$ and we have the corresponding estimate.
 \item The outer regularity property (see Section \ref{sectioncont}).
\end{itemize}

\subsection{Outer regularity and finiteness of the capacities $\gamma_n^1$ and $\gamma_n^2$}\label{sectionfiniteness}

Motivated by \cite{mpv} and \cite{cmpt}, we introduce now capacities related to only one kernel, $K_i=x_i^{2n-1}/|x|^{2n}$, $n\in\N$.  
For $1\le i\le 2$, we set

\begin{equation*}
\gamma_n^i(E)=\sup|\langle T,1\rangle|,
\end{equation*}
the supremum taken over those real distributions $T$ supported on $E$ such
that the potential $K_i*T$ is in the unit ball of $L^\infty(\Rd)$.

\noindent It is clear from the definition that for each compact set $E$, $1\le i\le 2$, $$\gamma_n(E)\le\gamma_n^i(E).$$
Little is known about these capacities $\gamma_n^i(E)$, because a growth condition like \eqref{basicgrowth} (see also \eqref{ourgrowth}) 
cannot be deduced from the $L^\infty-$boundedness of only one potential (see Section 5 of \cite{mpv} for some examples on this fact for the case $n=1$). 
We show that these capacities are finite and satisfy the outer regularity property. For this, we need the reproduction formula stated below.

\begin{lemma}\label{reproductionformula2}
If a function $f(x_1,x_2)$ has continuous derivatives up to order 2, then, for $1\le i\le 2$, it is representable in the form
\begin{equation}\label{repro2}
f(x)= (\varphi_i*K_i)(x),
\end{equation}
where \begin{equation}\label{primitive}\partial_i\varphi_i=S_i(\Delta f)=c\Delta f+\widetilde S_i(\Delta f),\end{equation}for some constant $c$ and the operators 
$S_i,\;\widetilde S_i$ as in Lemma \ref{reproductionformula}.

\end{lemma}

\begin{proof} Without loss of generality fix $i=1$.
By Lemma \ref{fourier}, we know that
\begin{equation}\label{form}
\widehat{K_1}(\xi)=c \frac{\xi_1}{|\xi|^{2n}}p(\xi_1,\xi_2),
\end{equation}
where $p(\xi_1,\xi_2)$ is a homogeneous polynomial of degree $2n-2$ with no non-vanishing zeros. 
Let $S_1$ be the operator with kernel 
$$\widehat{s_1}(\xi)=\frac{|\xi|^{2n-2}}{p(\xi_1,\xi_2)}.$$
By  \cite[Theorem 4.15, p.82]{duandikoetxea} (see also the proof of Lemma \ref{reproductionformula}), since the polynomial $p$ has no non-vanishing zeros, the operators $S_i$, $1\le i\le 2$, can be writen as $S_i=c\id+\widetilde S_i$, were $\widetilde S_1$ and $\widetilde S_2$  are Calder\'on-Zygmund operators. 

Now taking Fourier transforms on \eqref{primitive} with $i=1$, we obtain

$$\xi_1\widehat\varphi_1(\xi)=|\xi|^2\widehat f(\xi)\frac{|\xi|^{2n-2}}{p(\xi_1,\xi_2)},$$
which together with \eqref{form} gives

$$\widehat f(\xi)=c \widehat\varphi_1(\xi)\frac{\xi_1}{|\xi|^{2n}}p(\xi_1,\xi_2)=\widehat\varphi_1(\xi)\widehat K_1(\xi).$$
Therefore the lemma is proven.
\end{proof}

In \cite{mpv} it was shown that for a square $Q\subset\Rd$, the capacities $\gamma_1^i$, $1\le i\le 2$, satisfy $\gamma_1^i(Q)\le Cl(Q)$. 
We will now extend this result to the capacities $\gamma_n^i$, $1\le i\le 2$ and $n\ge 1$. 

\begin{lemma}
For any square $Q\subset\Rd$ and $1\le i\le 2$, we have $\displaystyle{\gamma_n^i(Q)\le Cl(Q).}$
\end{lemma}

\begin{proof}
Without loss of generality assume $i=1$. Let $T$ be a distribution supported on $Q$ such that the potential $K_1*T\in L^\infty(\Rd)$. 
Write $Q=I_1\times I_2$, with $I_j$, $1\le j\le 2$, being intervals in $\mathbb{ R }$, and let $\varphi_Q\in{\mathcal C}_0^\infty(2Q)$ be such that $\|\varphi_Q\|_\infty\le C$,
$\|\nabla\varphi_Q\|_\infty\le C\;l(Q)^{-1}$, $\|\nabla^2\varphi_Q\|_\infty\le C\;l(Q)^{-2}$ and $$\varphi_Q(x)=\varphi_1(x_1)\varphi_2(x_2),$$ with $\varphi_1(x_1)=1$ on $I_1$, $\varphi_1(x_1)=0$ on $(2I_1)^c$, $\int_{-\infty}^\infty\varphi_1=0$, $\varphi_2\geq 0$, $\varphi_2\equiv 1$ on $I_2$ and $\varphi_2\equiv 0$ on $(2I_2)^c$. 

Since our distribution $T$ is supported on $Q$, using \eqref{repro2} with $f$ and $\varphi_1$ replaced by $\varphi_Q$ and $\psi$ respectively, 
\begin{equation*}
|\langle T,1 \rangle|=|\langle T,\varphi_Q \rangle|=|\langle K_1*T, \psi \rangle|
\le \|K_1*T\|_\infty\|\psi\|_1,
\end{equation*}
where $\psi(x_1,x_2)=\int_{-\infty}^{x_1}\Delta\varphi_Q(t,x_2)dt+\int_{-\infty}^{x_1}\widetilde S_i(\Delta\varphi_Q)(t,x_2)dt$. Therefore, the lemma will be proven
 once we show that $\|\psi\|_1\le Cl(Q)$.\newline

Set $\psi_1(x_1,x_2)=\int_{-\infty}^{x_1}\Delta\varphi_Q(t,x_2)dt$. Notice that since the support of $\varphi_Q$ is $2Q$ and $\int_{-\infty}^\infty\varphi_1=0$, then 
the support of $\psi_1$ is also $2Q$ and writing $2I_1=[a,b]$, we get
$$\|\psi_1\|_1\le\|\partial_1\varphi_Q\|_1+\int_{2Q}|\partial_2^2\varphi_2(x_2)|\big|\int_{a}^{x_1}\varphi_1(t)dt\big|dx_1dx_2\le Cl(Q).$$

Set $\psi_2(x_1,x_2)=\int_{-\infty}^{x_1}\widetilde S_1(\Delta\varphi_Q)(t,x_2)dt$ and let $K(x)=K(x_1,x_2)$ be the kernel of $\widetilde S_1$. Then,
\begin{equation*}\begin{split}
\|\psi_2\|_1&=\int_{3Q}|\psi_2(x)|dx+\int_{(3Q)^c}|\psi_2(x)|dx\\\\&\le\int_{3Q}\left| \int_{-\infty}^{x_1} (K*\Delta\varphi_Q)(t,x_2)dt\right|dx+\int_{(3Q)^c}\left |\int_{-\infty}^{x_1} (K*\Delta\varphi_Q)(t,x_2)dt\right |dx\\\\&=A+B.
\end{split}
\end{equation*}

Recall that $Q=I_1\times I_2$ and  write $3I_1=[z_1,z_2]$. Then

\begin{equation*}
\begin{split}
B&=\int_{(3Q)^c}\left|\int_{-\infty}^{x_1}(K*\Delta\varphi_Q)(t,x_2)dt\right|dx\\\\&\le\int_{\begin{subarray}{l}(3Q)^c\\x_1<z_1\end{subarray}}\left|\int_{-\infty}^{x_1}(K*\Delta\varphi_Q)(t,x_2)dt\right|dx+\int_{\begin{subarray}{l}(3Q)^c\\x_1\in[z_1,z_2]\end{subarray}}\left|\int_{-\infty}^{x_1}(K*\Delta\varphi_Q)(t,x_2)dt\right|dx\\\\&+\int_{\begin{subarray}{l}(3Q)^c\\x_1>z_2\end{subarray}}\left|\int_{-\infty}^{x_1}(K*\Delta\varphi_Q)(t,x_2)dt\right|dx=B_1+B_2+B_3.
\end{split}
\end{equation*}

We deal now with $B_1$. By Fubini and standard estimates for the kernel of a Calder\'on-Zygmund operator we get

\begin{equation*}
\begin{split}
B_1&=\int_{\begin{subarray}{l}(3Q)^c\\x_1<z_1\end{subarray}}\left|\int_{-\infty}^{x_1}(K*\Delta\varphi_Q)(t,x_2)dt\right|dx\\\\&\le C \int_{\begin{subarray}{l}(3Q)^c\\x_1<z_1\end{subarray}}\int_{2Q}|\varphi_Q(w)|\int_{-\infty}^{x_1}\frac{dt}{|w-(t,x_2)|^4}dw\;dx\\\\&\le Cl(Q)^2\int_{\begin{subarray}{l}(3Q)^c\\x_1<z_1\end{subarray}}\int_{-\infty}^{x_1}\frac{dt}{|(t,x_2)|^4}dw\;dx.
\end{split}
\end{equation*}
Using that 
$$\int_{-\infty}^{x_1}\frac{dt}{|(t,x_2)|^4}\le\frac{C}{|x|^3},$$
we get
\begin{equation*}
B_1\le Cl(Q)^2\int_{(3Q)^c}\frac{dx}{|x|^3}\le Cl(Q)^2l(Q)^{-1}=Cl(Q).
\end{equation*}
Now we split $B_2$ in two terms:
\begin{equation*}
B_2=\int_{\begin{subarray}{l}(3Q)^c\end{subarray}}\left|\int_{-\infty}^{z_1}(K*\Delta\varphi_Q)(t,x_2)dt\right|dx+\int_{\begin{subarray}{l}(3Q)^c\\x_1\in[z_1,z_2]\end{subarray}}\left|\int_{z_1}^{x_1}(K*\Delta\varphi_Q)(t,x_2)dt\right|dx.
\end{equation*}
The first term above is $B_1$ with $x_1$ replaced by $z_1$. For the second term in $B_2$, say $B_{22}$, we use Tonelli and estimates for the kernel of a Calder\'on-Zygmund operator. Then we obtain
\begin{equation*}
\begin{split}
B_{22}&=\int_{\begin{subarray}{l}(3Q)^c\\x_1\in[z_1,z_2]\end{subarray}}\left|\int_{z_1}^{x_1}(K*\Delta\varphi_Q)(t,x_2)dt\right|dx\\\\&\le\int_{\begin{subarray}{l}(3Q)^c\\x_1\in[z_1,z_2]\end{subarray}}\int_{z_1}^{x_1}\int_{2Q}|\varphi_Q(w)||\Delta K(w-(t,x_2))|\;dw\;dt\;dx\\\\&\le C\;\int_{\begin{subarray}{l}(3Q)^c\\x_1\in[z_1,z_2]\end{subarray}}\int_{2Q}|\varphi_Q(w)|\int_{z_1}^{x_1}\frac{dt}{|w-(t,x_2)|^4}\;dw\;dx\\\\&\le Cl(Q)l(Q)^2l(Q)^{-2}=Cl(Q).
\end{split}
\end{equation*}

To deal with $B_3$, notice that since $\displaystyle{\int_{-\infty}^\infty \widetilde S_1(\Delta\varphi_Q)(t,x_2)dt=0}$, one has $$\int_{-\infty}^{x_1}(K*\Delta\varphi_Q)(t,x_2)dt=-\int_{x_1}^{\infty}(K*\Delta\varphi_Q)(t,x_2)dt,$$
so one argues as above.  

We are now left with the term $A$. Recall that $3I_1=[z_1,z_2]$ and write
\begin{equation*}
\begin{split}
A&=\int_{3Q}\left|\int_{-\infty}^{x_1}(K*\Delta\varphi_Q)(t,x_2)dt\right|dx\\\\&\le\int_{\begin{subarray}{l}3Q\\x_1<z_1\end{subarray}}\left|\int_{-\infty}^{x_1}(K*\Delta\varphi_Q)(t,x_2)dt\right|dx+\int_{\begin{subarray}{l}3Q\\x_1\in[z_1,z_2]\end{subarray}}\left|\int_{-\infty}^{x_1}(K*\Delta\varphi_Q)(t,x_2)dt\right|dx\\\\&+\int_{\begin{subarray}{l}3Q\\x_1>z_2\end{subarray}}\left|\int_{-\infty}^{x_1}(K*\Delta\varphi_Q)(t,x_2)dt\right|dx=A_1+A_2+A_3.
\end{split}
\end{equation*}
By Fubini and standard estimates for the kernel of a Calder\'on-Zygmund operator, we obtain
\begin{equation*}
\begin{split}
A_1&=\int_{\begin{subarray}{l}3Q\\x_1<z_1\end{subarray}}\left|\int_{-\infty}^{x_1}\int_{2Q}\Delta\varphi_Q(w)K(w-(t,x_2))dwdt\right|dx\\\\&
\le C \int_{\begin{subarray}{l}3Q\\x_1<z_1\end{subarray}}\int_{2Q}|\Delta\varphi_Q(w)|\int_{-\infty}^{x_1}\frac{dt}{|w-(t,x_2)|^2}dwdx\\\\&\le C\;\int_{3Q}\frac{dx}{|x|}\le Cl(Q).
\end{split}
\end{equation*}
Now we split $A_2$ in two terms
\begin{equation*}
\begin{split}
A_2&=\int_{\begin{subarray}{l}3Q\end{subarray}}\left|\int_{-\infty}^{z_1}(K*\Delta\varphi_Q)(t,x_2)dt\right|dx+\int_{\begin{subarray}{l}3Q\\x_1\in[z_1,z_2]\end{subarray}}\left|\int_{z_1}^{x_1}(K*\Delta\varphi_Q)(t,x_2)dt\right|dx\\\\&=A_{21}+A_{22}.
\end{split}
\end{equation*}
The term $A_{21}$ is treated as $A_1$ with $x_1$ replaced by $z_1$. For $A_{22}$, we use Tonelli, the Cauchy-Schwarz inequality and the fact that 
Calder\'on-Zygmund operators are bounded in $L^2$. Then we get,
\begin{equation*}
\begin{split}
A_{22}&\le\int_{3I_2}\int_{3I_1}\int_{3I_1}\left|\widetilde S_1(\Delta\varphi_Q)(t,x_2)\right|dt\;dx_1\;dx_2\\\\&=
\int_{3I_1}\int_{3I_2}\int_{3I_1}\left|\widetilde S_1(\Delta\varphi_Q)(t,x_2)\right|dt\;dx_2\;dx_1\\\\&\le C\;l(Q)\;\|\widetilde S_1(\Delta\varphi_Q)\|_{L^1(3Q)}\le C\;l(Q)^2\;\|\Delta\varphi_Q\|_2\le Cl(Q).
\end{split}
\end{equation*}
The estimate of $A_3$ is obtained similarly to $B_3$.
\end{proof}

As a consequence of the above result we have

\begin{co} For any compact set $E\subset\Rd$ and $1\le i\le 2$, $\displaystyle{\gamma_n^i(E)\le C\diam(E)}$.
\end{co}

We show now that the capacities $\gamma_n^i$, $1\le i\le 2$, satisfy the exterior regularity property, like the $\gamma_n$ (see Lemma \ref{extreg}).

\begin{lemma}\label{extreg2}
Let $\{E_k\}_k$ be a decreasing sequence of compact sets in $\Rd$, with intersection the compact set $E\subset\Rd$. 
Then, for $1\le i\le 2$, $\displaystyle{\gamma_n^i(E)=\lim_{k\to\infty}\gamma_n^i(E_k).}$
\end{lemma}

\begin{proof}
Without loss of generality set $i= 1$. Let us see that $\lim_{k\to\infty}\gamma^1_n(E_k)=\gamma^1_n(E)$. Clearly, the limit exists and $\lim_{k\to\infty}\gamma^1_n(E_k)\geq \gamma_n^1(E)$. To prove the converse inequality,
let $T_k$ be a distribution supported on $E_k$ such that $f_k=K_1*T_k$ is in the unit ball of $L^\infty(\Rd)$
and 
$$\gamma_n^1(E_k)-\frac1k < |\langle T_k,1\rangle|\leq \gamma_n^1(E_k).$$
By taking a subsequence if necessary, we may assume that $f_k$ converges weakly $*$ in $L^\infty(\Rd)$ to some function 
$f$ such that $\|f\|_\infty\leq 1$.

We will show that $T_k$ converges to some distribution $T$ such that $T*K_1$ is also  in the unit ball of $L^\infty(\Rd)$.
Then 
$$\gamma_n^1(E)\geq \langle T,\,1\rangle = \lim_{k\to\infty}\langle T_k,\,1\rangle = 
\lim_{k\to\infty}\gamma_n^1(E_k),$$ and we will be done.

Let us first check that the limit of $\{T_k\}_k$ exists in the topology of distributions. 
This is equivalent to saying that, for any $\vphi\in\CC_c^\infty(\Rd)$,
the limit $\lim_{k\to\infty}\langle T_k,\vphi\rangle$ exists.
To this end, let $u$ be a vector of the form $u=(u_1,0)$ such that
$$\supp(\vphi(\cdot - u)) \cap U_1(E)=\varnothing,$$
where $U_1(E)$ denotes the $1$-neighborhood of $E$. In this way, for $k$ big enough,
$$\langle T_k,\vphi\rangle = \langle T_k,\vphi- \vphi(\cdot - u)\rangle.$$
It is easy to check that there exists a function $\psi\in\CC_c^\infty(\Rd)$ such that $\partial_1\psi = \vphi-\vphi(\cdot - u)$. Then, using the reproducing formula \eqref{repro2}, we deduce that
\begin{align*}
\langle T_k,\vphi\rangle & = \langle T_k,\partial_1\psi\rangle = \langle T_k,\,S_1(\Delta \psi)*K_1\rangle = \langle T_k*K_1,\,S_1(\Delta \psi)\rangle,
\end{align*}
which is convergent, since $S_1(\Delta \psi)\in L^1(\Rd)$ arguing as in Lemma \ref{sublemma}, and $f_k = T_k*K_1$ is weak * convergent in 
$L^\infty(\Rd)$.

To see that $T*K_1$ is  in the unit ball of $L^\infty(\Rd)$, 
we take a radial function $\chi\in\CC^\infty(\Rd)$, $\int\chi=1$, supported in the unit ball and, as usual, we denote $\chi_\ve(x)
=\ve^{-2}\chi(\ve^{-1}x)$. Then 
it is enough to prove that 
$\chi_\ve*T*K_1$ is in the unit ball of $L^\infty(\Rd)$ for all $\ve>0$.
This follows easily: denoting $K_1^\ve = \chi_\ve*K_1$, for each $x\in\Rd$, we have
$$T_k* K_1^\ve(x) = \langle T_k,\,K_1^\ve(x-\cdot)\rangle.$$
Notice moreover that $\|T_k* K_1^\ve \|_\infty =\|\chi_\ve * (T_k*K_1)\|_\infty\le 1$.
Now, let $\psi_0\in\CC_c^\infty(\Rd)$ be such that it equals $1$ in $U_1(E)$, so that $T_k = \psi_0\,T_k$ for $k$ big enough. Then
$$T_k* K_1^\ve(x) = \langle \psi_0\,T_k,\,K_1^\ve(x-\cdot)\rangle = \langle T_k,\,\psi_0\,K_1^\ve(x-\cdot)\rangle,$$
which converges to $\langle T,\,\psi_0\,K_1^\ve(x-\cdot)\rangle = \langle T,\,K_1^\ve(x-\cdot)\rangle = T* K_1^\ve(x)$
as $k\to\infty$. Since $|T_k* K_1^\ve(x)|\leq 1$ for all $k$, we deduce that $|T* K_1^\ve(x)|\leq 1$ as wished, too.
\end{proof}

\begin{remark}
With little additional effort one can show that $T*K_1=f$ in the above proof.
\end{remark}

\end{document}